\documentclass[11 pts]{article}
\usepackage{eurosym}
\usepackage{hyperref}
\usepackage{amssymb}
\usepackage{amsfonts}
\usepackage{graphicx}
\usepackage{amsmath}
\usepackage{float,color,ulem}
\usepackage{epsf,epsfig,subfigure}
\usepackage{float,color,ulem}
\usepackage{ulem}
\usepackage{bbm}
\usepackage{authblk}

\newfloat{figure}{H}{lof}
\newfloat{table}{H}{lot}
\floatname{figure}{\figurename}
\floatname{table}{\tablename}
\newtheorem{theorem}{Theorem}[section]

\newtheorem{assumption}[theorem]{Assumption}

\newtheorem{corollary}[theorem]{Corollary}

\newtheorem{definition}[theorem]{Definition}

\newtheorem{lemma}[theorem]{Lemma}
\newtheorem{notation}[theorem]{Notation}

\newtheorem{proposition}[theorem]{Proposition}
\newtheorem{remark}[theorem]{Remark}

\hypersetup{colorlinks=true, linkcolor=blue, citecolor=red}
\newenvironment{proof}[1][Proof]{\textbf{#1.} }{\ \rule{0.5em}{0.5em}}

\numberwithin{equation}{section}
\allowdisplaybreaks
\begin{document}

\title{\textbf{Variation of constants formula and exponential dichotomy for non
autonomous non densely defined Cauchy problems}}
\author{ \textsc{Pierre Magal$^{a}$ and Ousmane Seydi$^{b}$} \\
$^{a}${\small Univ. Bordeaux, IMB, UMR 5251, F-33076 Bordeaux, France}\\
{\small CNRS, IMB, UMR 5251, F-33400 Talence, France}\\
$^{b}${\small D\'epartement Tronc Commun, \'{E}cole Polytechnique de Thi\`{e}s, S\'en\'egal}}
\maketitle

\textbf{Abstract:} {\small In this paper we prove a variation of constants formula for a non autonomous and non homogeneous Cauchy problems whenever the linear part is not densely defined and is not a Hille-Yosida operator. By using this variation of constants formula we derive a necessary and sufficient conditions for the existence of exponential dichotomy for the evolution family generated by the associated non autonomous homogeneous problem. We also prove a persistence result of the exponential dichotomy for small perturbations. Finally we illustrate our result by consider a parabolic equation with non local and non autonomous boundary conditions. } \vspace{0.3cm} 

\vspace{0.3cm}

\noindent \textbf{AMS Subject Classication} : 37B55, 47D62, 34D09

\tableofcontents

\section{Introduction}

In this article we consider the following non homogeneous non autonomous
problem 
\begin{equation}
\dfrac{du(t)}{dt}=(A+B(t))u(t)+f(t),\text{ for }t\geq t_{0},\text{ and }%
u(t_{0})=x\in \overline{D(A)},  \label{1.1}
\end{equation}%
where $t_{0}\in \mathbb{R}$, $A:D(A)\subset X\rightarrow X$ is a linear
operator (possibly with non dense domain i.e $\overline{D(A)}\varsubsetneq X$%
) on a Banach space $(X,\Vert \cdot \Vert )$, $\{B(t)\}_{t\in \mathbb{R}%
}\subset \mathcal{L}(\overline{D(A)},X)$ is a locally bounded and strongly
continuous family of bounded linear operators and $f\in L_{loc}^{1}(\mathbb{R%
},X)$.

Throughout this article, we will make the following assumption on the linear
operator $(A,D(A))$.

\begin{assumption}
\label{ASS1.1} We assume that

\begin{itemize}
\item[i)] There exist two constants $\omega\in \mathbb{R}$ and $M\geq 1$,
such that $(\omega,+\infty)\subset \rho (A)$ and 
\begin{equation*}
\left\Vert (\lambda I-A)^{-k}\right\Vert _{\mathcal{L}(\overline{D(A)})}\leq
M \left( \lambda -\omega \right) ^{-k},\;\forall \lambda >\omega,\;k\geq 1.
\end{equation*}

\item[ii)] $\lim_{\lambda \rightarrow +\infty }(\lambda I-A)^{-1}x=0,\
\forall x\in X$.
\end{itemize}
\end{assumption}

Recall that $A$ is a Hille-Yosida operator if there exist two constants $%
\omega \in \mathbb{R}$ and $M\geq 1$, such that $(\omega ,+\infty )\subset
\rho (A)$ and 
\begin{equation*}
\left\Vert (\lambda I-A)^{-k}\right\Vert _{\mathcal{L}(X)}\leq M\left(
\lambda -\omega \right) ^{-k},\forall \lambda >\omega ,\forall k\geq 1.
\end{equation*}%
In this article, we will not assume that $A$ is a Hille-Yosida operator since in Assumption \ref{ASS1.1}-\textit{i)} the operator norm is taken into $\overline{D(A)}$ instead of $X$. Set 
\begin{equation*}
X_{0}:=\overline{D(A)}
\end{equation*}%
and denote by $A_{0}$ the part of $A$ on $X_0$ that is 
\begin{equation*}
A_{0}x=Ax,\forall x\in D(A_{0})
\end{equation*}%
and 
\begin{equation*}
D(A_{0}):=\{x\in D(A):Ax\in X_{0}\}.
\end{equation*}%
Then it has been proved (see \cite[Lemma 2.1 and Lemma 2.2]{Magal-Ruan2009b}%
) Assumption \ref{ASS1.1} is equivalent to $\rho (A)\neq \varnothing $ and $%
A_{0}$ is a Hille-Yosida linear operator on $X_{0}$. Therefore $A_{0}$
generates a strongly continuous semigroup $\left\{ T_{A_{0}}(t)\right\}
_{t\geq 0}\subset \mathcal{L}(X_{0})$. An important and useful approach to
investigate such a non-densely defined Cauchy problems is to use the
integrated semigroup theory, which was first introduced by Arendt \cite%
{Arendt87-a, Arendt87-b}.\ The integrated semigroup generated by $A$, namely 
$\left\{ S_{A}(t)\right\} _{t\geq 0}$ is a strongly continuous familly of
bounded linear operator on $X,$ which commute with the resolvent of $A,$ and
such that for each $x\in X$ the map $t\rightarrow S_{A}(t)x$ is an
integrated solution of the Cauchy problem 
\begin{equation}
\frac{du}{dt}=Au(t)+x,\text{ for }t\geq 0\text{ and }u(0)=0.  \label{1.2}
\end{equation}%
Considering the Cauchy problem 
\begin{equation}
\frac{du}{dt}=Au(t)+f(t),\text{ for }t\geq 0\text{ and }u(0)=0,  \label{1.3}
\end{equation}%
with $f\in L^{1}\left( \left( 0,\tau \right) ,X\right) .\ $When $A$ is a
Hille-Yosida operator, it we was proved by Kellermann and Hieber \cite%
{Kellermann} that 
$$
t\rightarrow \left( S_{A}\ast f\right)
(t):=\int_{0}^{t}S_{A}(t-s)f(s)ds
$$ is continuously differential and the
derivative 
$$
u(t):=\frac{d}{dt}\left( S_{A}\ast f\right) (t)
$$ 
is an integrated (or mild) solution of (\ref{1.3}). That is to say 
\begin{equation*}
\int_0^tu(s)ds \in D(A), \forall t \geq 0,
\end{equation*}
and 
\begin{equation*}
u(t)=A\int_0^t u(s)ds+\int_0^t f(s)ds, \forall t \geq 0.
\end{equation*}
The uniqueness of mild solution has been proved by Thieme \cite{Thieme90-a}.

The next assumption needed to use perturbations $f(t)$ that are continuous
in time is the following.

\begin{assumption}
\label{ASS1.2} For each $\tau >0$ and each $f\in C\left( \left[ 0,\tau %
\right] ,X\right) $ we assume that there exists $u_{f}\in C\left( \left[
0,\tau \right] ,X_{0}\right) $ an integrated (or mild) solution of 
\begin{equation*}
\frac{du_{f}}{dt}=Au_{f}(t)+f(t),\text{ for }t\geq 0\text{ and }u_{f}(0)=0.
\end{equation*}
Moreover we assume that there exists a non decreasing map $\delta
:[0,+\infty) \rightarrow [0,+\infty)$ such that 
\begin{equation}  \label{1.4}
\Vert u_f(t) \Vert \leq \delta(t) \underset{s\in [0,t]}{\sup} \Vert
f(s)\Vert , \ \forall t\geq 0,
\end{equation}
with
\begin{equation*}
\delta(t)\rightarrow 0 \text{ as } t\rightarrow 0^+.
\end{equation*} 
%$\delta(t)\rightarrow 0$ as $t\rightarrow 0^+$.
\end{assumption}

Let $f \in C([0,+\infty),X)$ be fixed. The existence of mild solution in
Assumption \ref{ASS1.2} is equivalent to the continuous time
differentiability of the map $t\rightarrow (S_A \ast f)(t)$ from $%
[0,+\infty) $ into $X$. Moreover 
\begin{equation*}
u_f(t)=\frac{d}{dt}(S_A*f)(t), \forall t \geq 0,
\end{equation*}
whenever the mild solution exists. Define 
\begin{equation*}
(S_A \diamond f)(t) :=\frac{d}{dt}(S_A*f)(t), \forall t \geq 0.
\end{equation*}
The foregoing Assumption \ref{ASS1.2} needs justification. In fact if $A$ is
Hille-Yosida operator, then Assumption \ref{ASS1.2} holds true as long as $%
t\rightarrow f(t)$ is continuous (see Kellermann and Hieber \cite{Kellermann}%
) and we have the following estimate 
\begin{equation*}
\Vert (S_A \diamond f)(t) \Vert \leq M \displaystyle \int_0^t e^{\omega
(t-s)}\Vert f(s)\Vert ds
\end{equation*}
therefore 
\begin{equation}  \label{1.5}
\Vert (S_A \diamond f)(t) \Vert \leq \left(M \displaystyle \int_0^t
e^{\omega s} ds \right) \underset{s\in [0,t]}{\sup} \Vert f(s)\Vert , \
\forall t\geq 0,
\end{equation}
which clearly show that (\ref{1.4}) is a generalization of (\ref{1.5}).
Hence since we do not assume that $A$ is Hille-Yosida, Assumption \ref%
{ASS1.2} becomes an issue in order to obtain the existence of integrated
solutions for the non-homogeneous equation (\ref{1.1}). This question has
been studied by Magal and Ruan \cite{Magal-Ruan2007} and by Thieme \cite%
{Thieme08} in the general case, and by Ducrot, Magal and Prevost \cite{DMP}
in the almost sectorial case. Assumptions \ref{ASS1.1} and \ref{ASS1.2} are
justified by the fact that in several situations the linear operator $%
(A,D(A))$ is not Hille-Yosida while the differentiability of $t\rightarrow
(S_A\ast f)(t)$ and the estimate (\ref{1.4}) can be obtained if $%
(A_0,D(A_0)) $ is a Hille-Yosida operator.

The following assumption will be required in order to deal with the
existence of integrated solutions for the non homogeneous equation (\ref{1.1}%
).

\begin{assumption}
\label{ASS1.3} Let $\{B(t)\}_{t \in \mathbb{R}} \subset \mathcal{L}(X_0,X)$
be a family of bounded linear operators. We assume that $t \rightarrow B(t)$
is strongly continuous from $\mathbb{R}$ into $\mathcal{L}(X_0,X)$, that is
to say that for each $x \in X_0$ the map $t\rightarrow B(t)x$ is continuous
from $\mathbb{R}$ into $X$. We assume that there exists $b\in C(\mathbb{R},%
\mathbb{R}_+)$ such that 
\begin{equation*}
\Vert B(t)\Vert_{\mathcal{L}(X_0,X)} \leq b(t), \ \forall t\in \mathbb{R}.
\end{equation*}
\end{assumption}

The foregoing assumptions will allows us to obtain the existence of an \textit{%
evolution family} (see Definition \ref{DEF1.4} below) for the homogeneous
Cauchy problem (\ref{1.7}). Before proceeding let us introduce the notation 
\begin{equation*}
\Delta :=\left\{ (t,s)\in \mathbb{R}^{2}:t\geq s\right\} ,
\end{equation*}%
and recall the notion of evolution family.

\begin{definition}
\label{DEF1.4} Let $(Z,\Vert \cdot \Vert)$ be a Banach space. A two
parameters family of bounded linear operators on $Z$, $\{ U(t,s)\}_{(t,s)\in
\Delta}$ is an \textit{evolution family} if

\begin{itemize}
\item[i)] For each $t,r,s\in \mathbb{R}$ with $t\geq r\geq s$ 
\begin{equation*}
U(t,t)=I_{\mathcal{L}(Z)} \quad \text{and} \quad U(t,r)U(r,s)=U(t,s).
\end{equation*}

\item[ii)] For each $x\in Z$, the map $(t,s)\rightarrow U(t,s)x$ is
continuous from $\Delta$ into $Z$.
\end{itemize}

If in addition there exist two constants $\widehat{M}\geq 1$ and $\widehat{%
\omega}\in \mathbb{R}$ such that 
\begin{equation*}
\Vert U(t,s) \Vert_{\mathcal{L}(Z)} \leq \widehat{M} e^{\widehat{\omega}(t-s)}, \ \forall
(t,s)\in \Delta,
\end{equation*}
we say that $\{ U(t,s)\}_{(t,s)\in \Delta}$ is an exponentially bounded
evolution family.
\end{definition}

Consider the following homogeneous equation for each $t_0 \in \mathbb{R}$ 
\begin{equation}  \label{1.7}
\dfrac{du(t)}{dt}=(A+B(t))u(t), \text{ for } t\geq t_0 \text{ and } u(t_0)=x
\in X_0.
\end{equation}
By using \cite[Theorem 5.2]{Magal-Ruan2007} and \cite[Proposition 4.1]%
{Magal-Ruan2009a} we obtain the following Proposition.

\begin{proposition}
\label{PROP1.5} Let Assumptions \ref{ASS1.1}, \ref{ASS1.2} and \ref{ASS1.3}
be satisfied. Then the homogeneous Cauchy problem (\ref{1.7}) generates a
unique evolution family 
$\{U_B(t,s) \}_{(t,s)\in \Delta}\subset \mathcal{L}(X_0)$. Moreover 
$U_B(\cdot,t_0)x_0\in C([t_0,+\infty),X_0)$ is the unique solution of the fixed point problem 
\begin{equation}  \label{1.8}
U_B(t,t_0)x_0=T_{A_0}(t-t_0)x_0+\dfrac{d}{dt}\int_{t_0}^t
S_{A}(t-s)B(s)U_B(s,t_0)x_0ds, \ \forall t\geq t_0.
\end{equation}
If we assume in addition that 
\begin{equation*}
\sup_{t\in \mathbb{R}}  b(t)<+\infty
\end{equation*}
then the evolution family $\{U_B(t,s) \}_{(t,s)\in \Delta}$ is exponentially
bounded.
\end{proposition}

The following theorem provides an approximation formula of the solutions of equation (\ref{1.1}). This is the first main result. 

\begin{theorem}[Variation of constants formula]
\label{TH1.6} Let Assumptions \ref{ASS1.1}, \ref{ASS1.2} and \ref{ASS1.3} be
satisfied. Then for each $t_0\in \mathbb{R}$, each $x_0\in X_0$ and each $%
f\in C([t_0,+\infty],X)$ the unique integrated solution $u_f\in
C([t_0,+\infty],X_0)$ of (\ref{1.1}) is given by 
\begin{equation} \label{1.9}
u_f(t)=U_{B}(t,t_0)x_0+\underset{\lambda \rightarrow +\infty}{\lim}
\int_{t_0}^t U_B(t,s)\lambda R_\lambda(A)f(s)ds, \ \forall t\geq t_0,
\end{equation}
%where the limit 
%\begin{equation}  \label{1.9}
%\underset{\lambda \rightarrow +\infty}{\lim}\int_{t_0}^t U_B(t,s)\lambda
%R_\lambda(A)f(s)ds\in X_0,
%\end{equation}
%exists for each $t\geq t_0$. Moreover the convergence in (\ref{1.9}) is
%uniform with respect to $t,t_0 \in I$ for each compact interval $I \subset 
%\mathbb{R}$.
where the limit exists in $X_0$. Moreover the convergence in (\ref{1.9}) is
uniform with respect to $t,t_0 \in I$ for each compact interval $I \subset 
\mathbb{R}$.
\end{theorem}

Our second main result deal with a necessary and sufficient condition for
the evolution family (generated by the homogeneous problem associated to
system (\ref{1.1})) to have an exponential dichotomy. To be more precise let
us first recall some definitions and state our result.

\begin{definition}
\label{DEF1.7} Let $(Z,\Vert \cdot \Vert_Z)$ be a Banach space. We say that $%
\{\Pi(t)\}_{t\in \mathbb{R}} \subset \mathcal{L}(Z)$ is strongly continuous
family of projectors on $Z$ if 
\begin{equation*}
\Pi(t)\Pi(t)=\Pi(t), \ \forall t\in \mathbb{R},
\end{equation*}
and for each $x\in Z$, $t\rightarrow \Pi(t)x$ is continuous from $\mathbb{R}$
into $Z$.
\end{definition}

The following notion of exponential dichotomy will be used in this paper. We
refer for instance to \cite{ DMS2015a, DMS2015b, Rabiger, Hale-Lin, Henry,
Lian} and the references therein.

\begin{definition}
\label{DEF1.8} Let $(Z,\Vert \cdot \Vert_Z)$ be a Banach space. We say that
an evolution family $\lbrace U(t,s)\rbrace_{(t,s)\in \Delta}\subset \mathcal{%
L}(Z)$ has an exponential dichotomy with constant $\kappa\geq 1$ and
exponent $\beta>0$ if and only if the following properties are satisfied

\begin{itemize}
\item[i)] There exist two strongly continuous family of projectors $\lbrace
\Pi^+(t) \rbrace_{t\in \mathbb{R}}$ and $\lbrace \Pi^-(t)\rbrace_{t\in 
\mathbb{R}}$ on $Z$ such that 
\begin{equation*}
\Pi^+(t)+\Pi^-(t)=I_{\mathcal{L}(Z)}, \quad \forall t\in \mathbb{R}.
\end{equation*}
Then we define for all $t\geq s$ 
\begin{equation*}
U^{+}(t,s):=U(t,s)\Pi^{+}(s) \quad \text{and} \quad
U^{-}(t,s):=U(t,s)\Pi^{-}(s).
\end{equation*}

\item[ii)] For all $(t,s)\in \Delta$ we have $\Pi^{+}(t)U(t,s)=U(t,s)%
\Pi^{+}(s)$ and then $\Pi^{-}(t)U(t,s)=U(t,s)\Pi^{-}(s)$.

\item[iii)] For all $(t,s)\in \Delta$ the restricted linear operator $%
U(t,s)\Pi^{-}(s)$ is invertible from $\Pi^-(s)(Z)$ into $\Pi^-(t)(Z)$ with
inverse denoted by $\bar{U}^{-}(s,t)$ and we set 
\begin{equation*}
U^{-}(s,t):=\bar{U}^{-}(s,t)\Pi^{-}(t).
\end{equation*}

\item[iv)] For all $(t,s)\in \Delta$ 
\begin{equation*}
\Vert U^{+}(t,s) \Vert_{\mathcal{L}(Z)} \leq \kappa e^{-\beta (t-s)} \quad 
\text{ and } \quad \Vert U^{-}(s,t) \Vert_{\mathcal{L}(Z)} \leq \kappa e^{-\beta
(t-s)}.
\end{equation*}
\end{itemize}
\end{definition}

In the foregoing Definition \ref{DEF1.8} the notation $+$ and $-$ are used
to refer respectively the \textit{forward time} and the \textit{backward time%
}.

\begin{definition}
\label{DEF1.9} Let $f\in L^{1}_{loc}(\mathbb{R},X)$ be fixed. A function $u
\in C(\mathbb{R},X_0)$ is an integrated (or mild) solution of (\ref{1.1}) if
and only if for each $t \geq t_0$ 
\begin{equation*}
\int_{t_0}^{t} u(r) dr \in D(A)
\end{equation*}
and 
\begin{equation*}
u(t)=x+A\int_{t_0}^{t} u(r) dr +\int_{t_0}^{t}[B(r)u(r)+f(r)]dr.
\end{equation*}
\end{definition}

Then our second main result split into the following two theorems.

\begin{theorem}
\label{THEO1.10} Let Assumptions \ref{ASS1.1}, \ref{ASS1.2} and \ref{ASS1.3}
be satisfied. Assume in addition that 
\begin{equation*}
\sup_{t\in \mathbb{R}} b(t) <+\infty.
\end{equation*}
Then the following assertions are equivalent

\begin{itemize}
\item[i)] The evolution family $\{ U_B(t,s)\}_{(t,s)\in \Delta}$ has an
exponential dichotomy.

\item[ii)] For each $f\in BC(\mathbb{R},X)$, there exists a unique
integrated solution $u\in BC(\mathbb{R},X_0)$ of (\ref{1.1}).
\end{itemize}
\end{theorem}

\begin{theorem}
\label{THEO1.11} Let Assumptions \ref{ASS1.1}, \ref{ASS1.2} and \ref{ASS1.3}
be satisfied. Assume in addition that 
\begin{equation*}
\sup_{t\in \mathbb{R}} b(t) <+\infty.
\end{equation*}
If $U_B$ has an exponential dichotomy with exponent $\beta>0$, then for each 
$\eta \in [0,\beta)$ and each $f\in BC^\eta(\mathbb{R},X)$ with 
\begin{equation*}
BC^\eta(\mathbb{R},X):=\left \lbrace f \in C(\mathbb{R},Z): \Vert f
\Vert_{\eta}:= \underset{t\in \mathbb{R}}{\sup }e^{-\eta \vert t \vert}\Vert
f(t) \Vert_Z<+\infty \right \rbrace
\end{equation*}
there exists a unique integrated solution $u\in BC^\eta(\mathbb{R},X_0)$ of (%
\ref{1.1}) which is given by 
\begin{equation}  \label{1.10}
u_f(t)=\underset{\lambda \rightarrow +\infty}{\lim} \int_{-\infty}^{t}
U^+_B(t,s)\lambda R_\lambda(A) f(s)ds-\int_{t}^{+\infty} U^-_B(t,s)\lambda
R_\lambda(A) f(s)ds, \ \forall t\in \mathbb{R}.
\end{equation}
Moreover the following properties hold true 

\begin{itemize}
\item[i)] The limit (\ref{1.10}) exists uniformly on compact subset of $%
\mathbb{R}$.

\item[ii)] If $f$ is bounded and uniformly continuous with relatively
compact range then the limit (\ref{1.10}) is uniform on $\mathbb{R}$.

\item[iii)] For each $\nu\in (-\beta,0)$ there exists $C(\nu,\kappa,\beta)>0$
such that 
\begin{equation*}
\Vert u_f \Vert_\eta \leq C(\nu,\kappa,\beta) \Vert f \Vert_\eta, \ \forall
\eta \in [0, -\nu].
\end{equation*}
\end{itemize}
\end{theorem}
Let us mention that similar results have been investigated in \cite{Rabiger} where $(A,D(A))$ is a Hille-Yosida linear operator in the context of extrapolated semigroups. However the fact that $(A,D(A))$ is not Hille-Yosida induces several difficulties that need new technical  arguments in order to prove our more general results that  by using integrated semigroups theory.\\

The paper is organized has follow. In section \ref{Sect-Preliminaries} we
recall some results concerning integrated semigroups and define the notion
of integrated solution for system (\ref{1.1}). Section \ref{Sect-Variation
of constants} is devoted to the the proof of Theorem \ref{TH1.6} concerning
the variation of constants formula. In Section \ref{Sect-Uniform convergence}
we prove some uniform convergence results. Finally Theorems \ref{THEO1.10}
and \ref{THEO1.11} are proved in Section \ref{Sect-Exponential dichotomy}.

\section{Preliminaries}

\label{Sect-Preliminaries} In the subsequent lemma we summarize some results
proved in Magal and Ruan \cite[Lemma 2.1 and Lemma 2.2]{Magal-Ruan2009b}.
For more simplicity in the notations we define 
\begin{equation*}
R_\lambda(A):=(\lambda I-A)^{-1}, \ \forall \lambda>\omega.
\end{equation*}

\begin{lemma}
\label{LE2.1} Let Assumption \ref{ASS1.1} be satisfied. Then we have 
\begin{equation*}
\rho (A)=\rho (A_{0}).
\end{equation*}%
Moreover, we have the following properties

\begin{itemize}
\item[i)] For each $\lambda >\omega $ 
\begin{equation*}
D(A_{0})=R_{\lambda }(A)(X_{0})\ \text{ and }\ R_{\lambda }(A)_{\mid
X_{0}}=R_{\lambda }(A_{0}).
\end{equation*}

\item[ii)] $\overline{D(A_{0})}=X_{0}$.

\item[iii)] $\underset{\lambda \rightarrow +\infty }{\lim }R_{\lambda
}(A)x=x,\ \forall x\in X_{0}$.
\end{itemize}
\end{lemma}

\begin{remark}
It can be easily proved that $\underset{\lambda \rightarrow +\infty }{\lim }%
R_{\lambda }(A)x=x$ uniformly for $x$ in a relatively compact subset of $%
X_{0}$. This property we will be often used in this paper.
\end{remark}

Note that if $(A,D(A))$ satisfies Assumption \ref{ASS1.1} then by Lemma \ref%
{LE2.1} we have 
\begin{equation*}
\left\Vert R_{\lambda }(A_0)^{k}\right\Vert _{\mathcal{L}(X_0)}\leq M \left(
\lambda -\omega \right) ^{-k},\;\forall \lambda >\omega,\;k\geq 1 \ \text{%
and } \overline{D(A_0)}=X_0.
\end{equation*}
Therefore $(A_0,D(A_0))$ generates a strongly continuous semigroup $\lbrace
T_{A_0}(t) \rbrace_{t\geq 0}\subset \mathcal{L}(X_0)$ with 
\begin{equation*}
\Vert T_{A_0}(t)\Vert_{\mathcal{L}(X_0)}\leq M e^{\omega t}, \quad \forall
t\geq 0.
\end{equation*}
The characterization of an integrated semigroup is summarized in the
definition below.

\begin{definition}
Let $(X,\Vert \cdot \Vert)$ be a Banach space. A family of bounded linear
operators $\{S(t)\}_{t\geq 0}$ on $X$ is called an integrated semigroup if

\begin{itemize}
\item[i)] $S(0)x=0, \forall x\in X$.

\item[ii)] $t\rightarrow S(t)x$ is continuous on $[0,+\infty)$ for each $%
x\in X$.

\item[iii)] For each $t\geq 0$, $S(t)$ satisfies 
\begin{equation*}
S(s)S(t)=\int_0^s [S(r+t)-S(r)]dr, \ \forall s\geq 0.
\end{equation*}
\end{itemize}

The integrated semigroup $\{S(t)\}_{t\geq 0}$ is said non-degenerate if 
\begin{equation*}
S(t)x=0, \ \forall t\geq 0 \Rightarrow x=0.
\end{equation*}
Moreover we will say that $(A,D(A))$ generates an integrated semigroup $%
\lbrace S_A(t)\rbrace_{t\geq 0} \subset \mathcal{L}(X,X_0)$ that is 
\begin{equation*}
x\in D(A) \ \text{and}\ y=Ax \Leftrightarrow S_A(t)x=tx+\int_0^t S(s)yds, \
\forall t\geq 0.
\end{equation*}
\end{definition}

The following result is well known in the context of integrated semigroups.

\begin{proposition}
Let Assumption \ref{ASS1.1} be satisfied. Then $(A,D(A))$ generates a
uniquely determined non-degenerate exponentially bounded integrated
semigroup with 
\begin{equation*}
\Vert S_A(t) \Vert_{\mathcal{L}(X)} \leq \hat{M}e^{\hat{\omega}t},
\end{equation*}
where $\hat{M}> 0$, $\hat{\omega}>0$ and $(\hat{\omega},+\infty)\in \rho(A)$.%
\newline
Moreover the following properties hold

\begin{itemize}
\item[i)] For each $x\in X$, each $t\geq 0$, each $\mu >\omega$, $S_A(t)x$
is given by 
\begin{equation*}
S_A(t)x=(\lambda I- A)\int_0^t T_{A_0}(s)ds (\lambda I- A)^{-1}, \ \forall
\lambda>\omega,
\end{equation*}
or equivalently 
\begin{equation*}
S_A(t)x=\mu \int_0^t T_{A_0}(s) R_\mu(A)xds+[I-T_{A_0}(t)]R_\mu(A)x.
\end{equation*}

\item[ii)] The map $t\rightarrow S_A(t)x$ is continuously differentiable if
and only if $x\in X_0$ and 
\begin{equation*}
\dfrac{dS_A(t)x}{dt}=T_{A_0}(t)x, \ \forall t\geq 0, \ \forall x\in X_0.
\end{equation*}
\end{itemize}
\end{proposition}

Next we give the notion of integrated solution for system (\ref{1.1}).

\begin{definition}
Let $t_0 \in \mathbb{R}$ and let $f\in L^{1}_{loc}((t_0,+\infty),X)$ be
fixed. A function $u \in C([t_0,+\infty),X_0)$ is an integrated (or mild)
solution of (\ref{1.1}) if and only if for each $t \geq t_0$ 
\begin{equation*}
\int_{t_0}^{t} u(r) dr \in D(A)
\end{equation*}
and 
\begin{equation*}
u(t)=x+A\int_{t_0}^{t} u(r) dr +\int_{t_0}^{t}[B(r)u(r)+f(r)]dr.
\end{equation*}
\end{definition}

The following result is a direct consequence of Theorem 2.10 in \cite%
{Magal-Ruan2009a}.

\begin{theorem}
\label{TH2.11} Let Assumptions \ref{ASS1.1} and \ref{ASS1.2} be satisfied.
Let $t_0\in \mathbb{R}$ be fixed. Then for all $f\in C([t_0,+\infty),X)$,
the map $t\rightarrow (S_A \ast f(t_0+\cdot))(t-t_0)$ is continuously
differentiable from $[t_0,+\infty)$ into $X$ and satisfies the following
properties

\begin{itemize}
\item[i)] $(S_A \ast f(t_0+\cdot))(t-t_0) \in D(A)$, $\forall t\geq t_0$.

\item[ii)] If we set 
\begin{equation*}
u(t):=(S_A \diamond f(t_0+\cdot))(t-t_0), \ \forall t\geq t_0,
\end{equation*}
then the following hold 
\begin{equation*}
u(t)=A\int_{t_0}^t u(s)ds+\int_{t_0}^t f(s)ds, \ \forall t\geq t_0,
\end{equation*}
and 
\begin{equation*}
\Vert u(t)\Vert \leq \delta(t-t_0) \underset{s\in [t_0,t]}{\sup}\Vert f(s)
\Vert, \ \forall t\geq t_0.
\end{equation*}

\item[iii)] For all $\lambda \in (\omega ,+\infty)$ we have for each $t\geq
t_0$ 
\begin{eqnarray*}
R_\lambda(A) \dfrac{d}{dt}(S_A \ast f(t_0+\cdot))(t-t_0)=\int_{t_0}^t
T_{A_0}(t-s) R_\lambda(A) f(s)ds.
\end{eqnarray*}
\end{itemize}
\end{theorem}

As a consequence of \textit{iii)} in Theorem \ref{TH2.11}, we obtain the
following approximation formula 
\begin{equation}  \label{2.1}
\dfrac{d}{dt}\int_{t_0}^t S_A(t-s)f(s)ds=\underset{\lambda\rightarrow+\infty}%
{\lim} \int_{t_0}^t T_{A_0}(t-s)\lambda R_\lambda(A) f(s)ds, \ \forall t\geq
t_0.
\end{equation}
It also follows that for each $t,h \geq 0$ 
\begin{equation}
(S_A \diamond f)(t+h)=T_{A_0}(h)(S_A \diamond f)(t) +(S_A \diamond
f(t+\cdot))(h).  \label{2.2}
\end{equation}
As an immediate consequence of Theorem \ref{TH2.11} we obtain the following
lemma.

\begin{lemma}
\label{LE2.12} Let Assumptions \ref{ASS1.1} and \ref{ASS1.2} be satisfied.
Let $f\in C(\mathbb{R},X)$. Then the map $(t,t_0) \rightarrow (S_A \diamond
f(t_0+\cdot))(t-t_0)$ is continuous from $\Delta$ into $X$.
\end{lemma}

\begin{proof}
Let $(t,t_0),(s,s_0) \in \Delta$. We have 
\begin{equation*}
\begin{array}{ll}
I & :=(S_A \diamond f(t_0+\cdot))(t-t_0)-(S_A \diamond f(s_0+\cdot))(s-s_0)
\\ 
& =(S_A \diamond \big[ f(t_0+\cdot)-f(s_0+\cdot) \big] )(t-t_0) \\ 
& +(S_A \diamond f(s_0+\cdot))(t-t_0)-(S_A \diamond f(s_0+\cdot))(s-s_0)%
\end{array}%
\end{equation*}
hence by using (\ref{2.2}) 
\begin{equation*}
\begin{array}{ll}
I & =(S_A \diamond \big[ f(t_0+\cdot)-f(s_0+\cdot) \big] )(t-t_0) \\ 
& +\big[T_{A_0}((t-t_0)-(s-s_0))-I \big](S_A \diamond f(s_0+\cdot))(s-s_0)
\\ 
& +(S_A \diamond f(s_0+(s-s_0)+\cdot))((t-t_0)-(s-s_0))%
\end{array}%
\end{equation*}
whenever $t-t_0 \geq s-s_0$. The result follows by using the uniform
continuity of $f$ on bounded intervals.
\end{proof}

By using \cite[Proposition 4.1]{Magal-Ruan2009a} we obtain the following
lemma.

\begin{lemma}
\label{LE2.13} Let Assumptions \ref{ASS1.1}, \ref{ASS1.2} and \ref{ASS1.3}
be satisfied. Let $t_0\in \mathbb{R}$ be fixed. Then for each $x_0\in X_0$
and $f\in C([t_0,+\infty),X)$ there exists a unique integrated solution $%
u_f\in C([t_0,+\infty),X_0)$ of (\ref{1.1}) given by 
\begin{equation*}
u_f(t)=T_{A_0}(t-t_0)x_0+\dfrac{d}{dt}(S_A \ast
((Bu_f)(t_0+\cdot)+f(t_0+\cdot))(t-t_0), \ \forall t\geq t_0,
\end{equation*}
or equivalently 
\begin{equation*}
u_f(t)=T_{A_0}(t-t_0)x_0+(S_A \diamond
((Bu_f)(t_0+\cdot)+f(t_0+\cdot))(t-t_0), \ \forall t\geq t_0,
\end{equation*}
where we have used the notation $(Bu_f)(t):=B(t)u_f(t)$ for every $t\geq t_0$%
.
\end{lemma}

The next result is due to Magal and Ruan \cite[Proposition 2.14]%
{Magal-Ruan2009a} and is one of the main tools in studying integrated
solution for non Hille-Yosida operators. It reads as

\begin{proposition}
\label{PROP2.16} Let Assumption \ref{ASS1.1} be satisfied. Let $%
\varepsilon>0 $ be given and fixed. Then, for each $\tau_\varepsilon>0$
satisfying $M \delta(\tau_\varepsilon)\leq \varepsilon$, we have 
\begin{equation*}
\Vert \dfrac{d}{dt}(S_A\ast f)(t)\Vert \leq C(\varepsilon,\gamma) \underset{%
s\in [0,t]}{\sup} e^{\gamma (t-s)}\Vert f(s)\Vert, \ \forall t\geq 0,
\end{equation*}
whenever $\gamma \in (\omega,+\infty)$, $f\in C(\mathbb{R}_+,X)$ with 
\begin{equation*}
C(\varepsilon,\gamma):=\dfrac{2 \varepsilon \max (1,e^{-\gamma
\tau_\varepsilon })}{1-e^{(\omega-\gamma) \tau_\varepsilon}}.
\end{equation*}
\end{proposition}

\section{A variation of constants formula}

\label{Sect-Variation of constants} In this section we will prove the first
main result of this paper. It deal with the representation of the integrated
solution of (\ref{1.1}) in term of the evolution family $\lbrace
U_B(t,s)\rbrace_{(t,s)\in \Delta}$. This result generalize \cite[Theorem 2.2]%
{Rabiger} to the context of non Hille-Yosida operator. The proof will be
given by using several technical lemmas. Note that a direct consequence of
Theorem \ref{TH1.6} is the following

\begin{corollary}
\label{COR3.1} Let Assumptions \ref{ASS1.1}, \ref{ASS1.2} and \ref{ASS1.3}
be satisfied. Then for each $t_0\in \mathbb{R}$, each $x_0\in X_0$ and each $%
f\in C([t_0,+\infty],X_0)$ the unique integrated solution $u_f\in
C([t_0,+\infty],X_0)$ of (\ref{1.1}) is given by 
\begin{equation*}
u_f(t)=U_{B}(t,t_0)x_0+\int_{t_0}^t U_B(t,s)f(s)ds, \ \forall t\geq t_0.
\end{equation*}
\end{corollary}

Next we prove some technical lemmas that will be crucial for the proof of
Theorem \ref{TH1.6}.

\begin{lemma}
\label{LE3.2} Let Assumptions \ref{ASS1.1}, \ref{ASS1.2} and \ref{ASS1.3} be
satisfied. Then for each $h\in C(\Delta,X)$ the following equality holds 
\begin{equation*}
\begin{array}{ll}
\displaystyle\int_{t_0}^t \dfrac{d}{dt} \left[ \int_{s}^t S_{A}(t-r) h(r,s)
dr \right] ds= &  \\ 
\ \ \ \ \ \ \ \ \ \ \ \ \ \ \ \ \ \ \ \ \ \ \ \ \ \ \ \ \ \ \ \ %
\displaystyle \dfrac{d}{dt} \int_{t_0}^t S_{A}(t-r)\left[\int_{t_0}^r h(r,s)
ds\right] dr, & 
\end{array}%
\end{equation*}
for all $(t, t_0)\in \Delta$.
\end{lemma}

\begin{proof}
Let $t_0\in \mathbb{R}$ be fixed. Let $s\geq t_0$ be given. Then observing
that $h(\cdot,s) \in C([s,+\infty),X)$ one can apply Theorem \ref{TH2.11} to
obtain for all $t\geq s$ and $\lambda>\omega$ 
\begin{equation}  \label{3.1}
\int_{s}^t T_{A_0}(t-r) \lambda R_{\lambda} (A) h(r,s) dr=\lambda
R_{\lambda} (A) \dfrac{d}{dt} \int_{s}^t S_{A}(t-r) h(r,s) dr.
\end{equation}
Thus integrating the both sides of (\ref{3.1}) and using Fubini's theorem we
obtain for each $t\geq t_0$ and $\lambda>\omega$ 
\begin{equation*}
\begin{array}{ll}
\displaystyle \lambda R_{\lambda} (A) \int_{t_0}^t \left[\dfrac{d}{dt}
\int_{s}^t S_{A}(t-r) h(r,s) dr\right]ds & = \displaystyle\int_{t_0}^t \left[%
\int_{s}^t T_{A_0}(t-r) \lambda R_{\lambda} (A) h(r,s) dr \right]ds \\ 
& = \displaystyle\int_{t_0}^t \left[\int_{t_0}^r T_{A_0}(t-r) \lambda
R_{\lambda} (A) h(r,s) ds \right]dr \\ 
& =\displaystyle\int_{t_0}^t T_{A_0}(t-r) \lambda R_{\lambda} (A) \left[%
\int_{t_0}^r h(r,s) ds \right]dr.%
\end{array}%
\end{equation*}
Now observing that 
\begin{equation*}
\int_{t_0}^t \left[\dfrac{d}{dt} \int_{s}^t S_{A}(t-r) h(r,s) dr\right]ds\in
X_0, \ \forall t\geq t_0,
\end{equation*}
the result follows since we have 
\begin{equation*}
\underset{\lambda\rightarrow +\infty}{\lim}\lambda R_{\lambda} (A)
\int_{t_0}^t \left[\dfrac{d}{dt} \int_{s}^t S_{A}(t-r) h(r,s) dr\right]%
ds=\int_{t_0}^t \left[\dfrac{d}{dt} \int_{s}^t S_{A}(t-r) h(r,s) dr\right]ds,
\end{equation*}
for all $t\geq t_0$ and (see equality (\ref{2.1})) 
\begin{equation*}
\underset{\lambda \rightarrow +\infty}{\lim} \displaystyle\int_{t_0}^t
T_{A_0}(t-r) \lambda R_{\lambda} (A) \left[\int_{t_0}^r h(r,s) ds \right]dr= 
\dfrac{d}{dt} \int_{t_0}^t S_{A}(t-r)\left[\int_{t_0}^r h(r,s) ds\right] dr,
\end{equation*}
for all $t\geq t_0$.
\end{proof}

Using Lemma \ref{LE3.2} and Proposition \ref{PROP2.16} we can prove the
following technical lemma.

\begin{lemma}
\label{LE3.3} Let Assumptions \ref{ASS1.1}, \ref{ASS1.2} and \ref{ASS1.3} be
satisfied. Let $f\in C(\mathbb{R},X)$. Define for each $\lambda >\omega $
and $(t,t_0)\in \Delta$ 
\begin{equation*}
v_\lambda (t,t_0):=\int_{t_0}^t U_B(t,s)\lambda R_\lambda(A)f(s)ds,
\end{equation*}
and 
\begin{equation*}
w(t,t_0):=\dfrac{d}{dt}\int_{t_0}^t S_A(t-s) f(s) ds=(S_A\diamond
f(t_0+\cdot))(t-t_0).
\end{equation*}
Then we have the following properties

\begin{itemize}
\item[i)] For each $\lambda>\omega $ and $(t,t_0)\in \Delta$ 
\begin{equation*}
v_\lambda(t,t_0)=\dfrac{d}{dt} \int_{t_0}^t S_{A}(t-r) B(r)v_\lambda(r,t_0)
dr+\lambda R_\lambda(A)w(t,t_0) , \ \forall t\geq t_0.
\end{equation*}

\item[ii)] If in addition $\sup_{t\in \mathbb{R}}  b(t) <+\infty$
then there exists a constant $\gamma>\max(0,\omega)$ such that for each $%
\lambda >\omega$ and $(t,t_0)\in \Delta$ 
\begin{equation*}
\underset{ s \in [t_0,t] }{\sup} e^{-\gamma s} \Vert v_\lambda(s,t_0)\Vert
\leq 2\ \underset{ s\in [t_0,t] }{\sup} e^{-\gamma s} \Vert \lambda
R_\lambda(A)w(s,t_0)\Vert
\end{equation*}
and since $w(s,t_0) \in X_0$ we have 
\begin{equation*}
\Vert \lambda R_\lambda(A)w(s,t_0)\Vert \leq \dfrac{M \vert \lambda \vert }{%
\lambda-\omega} \Vert w(s,t_0) \Vert , \ \forall s \in [t_0,t].
\end{equation*}
\end{itemize}
\end{lemma}

\begin{proof}
\textit{Proof of \textit{i)} :} By using formula (\ref{1.8}) we obtain for
each $\lambda>\omega$ and $t\geq t_0$ 
\begin{equation*}
\begin{array}{ll}
v_\lambda (t,t_0)= & \displaystyle \int_{t_0}^t T_{A_0}(t-s)\lambda
R_\lambda(A)f(s) ds \\ 
& \displaystyle+\int_{t_0}^t \left[\dfrac{d}{dt}\int_{s}^t
S_{A}(t-r)B(r)U_B(r,s)\lambda R_\lambda(A)f(s)dr\right] ds.%
\end{array}%
\end{equation*}
Now note that from Theorem \ref{TH2.11} we have for each $\lambda>\omega$ 
\begin{equation}  \label{3.2}
\displaystyle \int_{t_0}^t T_{A_0}(t-s)\lambda R_\lambda(A)f(s) ds= \lambda
R_\lambda(A) \dfrac{d}{dt}\int_{t_0}^t S_A(t-s) f(s) ds, \ \forall t\geq t_0,
\end{equation}
and from Lemma \ref{LE3.2} with $h(r,s)=B(r)U_B(r,s)\lambda R_\lambda(A)f(s)$
\begin{equation}  \label{3.3}
\int_{t_0}^t \left[\dfrac{d}{dt}\int_{s}^t S_{A}(t-r)h(r,s)dr\right] ds= 
\dfrac{d}{dt} \int_{t_0}^t S_{A}(t-r) v_\lambda(r,t_0) dr, \ \forall t\geq t_0.
\end{equation}
Then \textit{i)} follows by combining (\ref{3.2}) and (\ref{3.3}).\newline
\textit{Proof of \textit{ii)} :} To do this we will make use of Proposition %
\ref{PROP2.16}. Let $\varepsilon>0$ be given such that 
\begin{equation}  \label{3.4}
2\varepsilon \ \underset{s\in \mathbb{R}}{\sup} b(s) <\frac{1}{4}.
\end{equation}
Let $\tau_\varepsilon >0$ be given with $M\delta (\tau_\varepsilon)\leq
\varepsilon$. By combining Proposition \ref{PROP2.16} together with \textit{%
i)} we obtain for each $\lambda > \omega$ and $t\geq t_0$ that 
\begin{equation*}
\Vert v_\lambda(t,t_0)\Vert \leq C(\varepsilon,\gamma) \underset{s\in
[t_0,t] }{\sup} \left[ e^{\gamma (t-s)} \ b(s) \ \Vert v_\lambda(s,t_0)\Vert %
\right]+ \Vert \lambda R_\lambda(A)w(t,t_0)\Vert,
\end{equation*}
whenever $\gamma \in (\omega,+\infty)$ with 
\begin{equation}  \label{3.5}
C(\varepsilon,\gamma):=\dfrac{2 \varepsilon \max (1,e^{-\gamma
\tau_\varepsilon })}{1-e^{(\omega-\gamma) \tau_\varepsilon}},
\end{equation}
so that 
\begin{equation*}
\underset{s\in [t_0,t] }{\sup} e^{-\gamma s} \Vert v_\lambda(s,t_0)\Vert
\leq C(\varepsilon,\gamma) \underset{s\in \mathbb{R}}{\sup}\ b(s)\ \underset{%
s\in [t_0,t] }{\sup} e^{-\gamma s}\Vert v_\lambda(s,t_0)\Vert +\underset{%
s\in [t_0,t] }{\sup}\Vert \lambda R_\lambda(A)w(s,t_0)\Vert.
\end{equation*}
By using (\ref{3.5}) and (\ref{3.4}) it is easily seen that one can chose $%
\gamma>\max(0,\omega)$ large enough such that 
\begin{equation*}
0\leq C(\varepsilon,\gamma)\underset{s\in \mathbb{R}}{\sup} \ b(s) <\frac{1}{%
2},
\end{equation*}
and \textit{ii)} follows.
\end{proof}

The following Lemma will be needed in sequel.

\begin{lemma}
\label{LE3.4} Let Assumptions \ref{ASS1.1} and \ref{ASS1.2} be satisfied.
Then for each $a,c\in \mathbb{R}$ with $a<c$ and each $x\in X$, the map $%
t\rightarrow (S_A \ast x \mathbbm{1}_{[a,c)}(\cdot))(t)$ is differentiable
on $[0,+\infty)$ and 
\begin{equation*}
\dfrac{d}{dt}(S_A \ast x \mathbbm{1}_{[a,c)}(\cdot))(t)= \left\lbrace 
\begin{array}{ll}
0 & \text{ if } c\leq 0 \text{ or } \ t \leq a, \\ 
S_A(t-a^+)x & \text{ if } c>0 \text{ and } t\in [a ,c) , \\ 
T_{A_0}(t-c)S_A(c-a^+)x & \text{ if } c>0 \text{ and } t\geq c, \\ 
& 
\end{array}
\right.
\end{equation*}
with $a^+:=\max(0,a)$.
\end{lemma}

\begin{proof}
The proof is straightforward.
\end{proof}

\vspace{0.5cm}

Now we have all the material in order to prove Theorem \ref{TH1.6}. \vspace{%
0.5cm}\newline

\begin{proof}[Proof of Theorem \protect\ref{TH1.6}]
Since the proof is trivial when $f(t)=0$ it is sufficient to prove our
theorem for $x_0=0$. Let $t_0\in \mathbb{R}$ be fixed. Recalling for each $%
\lambda> \omega$ 
\begin{equation*}
v_\lambda (t,t_0)=\int_{t_0}^t U_B(t,s)\lambda R_\lambda(A)f(s)ds, \ \forall
t\geq t_0,
\end{equation*}
we will show that the limit 
\begin{equation}  \label{3.6}
\bar{v}(t,t_0):= \underset{\lambda\rightarrow +\infty}{\lim}
v_\lambda(t,t_0), \ \forall t\geq t_0,
\end{equation}
is well defined and is an integrated solution of 
\begin{equation}  \label{3.7}
\dfrac{dv(t)}{dt}=[A+B(t)]v(t)+f(t), \ t\geq t_0 \ \text{and} \ v(t_0)=0.
\end{equation}
First of all note that by Lemma \ref{LE2.13}, problem (\ref{3.7}) admits a
unique integrated solution $v(\cdot,t_0)\in C([t_0,+\infty),X_0)$ satisfying 
\begin{equation}  \label{3.8}
v(t,t_0)=(S_A \diamond (Bv(\cdot,t_0))(t_0+\cdot))(t-t_0)+ (S_A \diamond
f(t_0+\cdot))(t-t_0), \ \forall t\geq t_0,
\end{equation}
where we have used the notation $(Bv(\cdot,t_0))(t)=B(t)v(t,t_0)$ for every $%
t\geq t_0$. Furthermore by Lemma \ref{LE3.3} we also have for each $%
\lambda>\omega$ and each $t\geq t_0$ 
\begin{equation}  \label{3.9}
v_\lambda(t,t_0)=\dfrac{d}{dt} \int_{t_0}^t S_{A}(t-r) B(r)v_\lambda(r,t_0)
dr+\lambda R_\lambda(A)w(t,t_0) , \ \forall t\geq t_0,
\end{equation}
with 
\begin{equation}
w(t,t_0)=\dfrac{d}{dt}\int_{t_0}^t S_A(t-s) f(s) ds= (S_A \diamond
f(t_0+\cdot))(t-t_0), \ \forall t\geq t_0.
\end{equation}
Then (\ref{3.8}) and (\ref{3.9}) rewrites, for each $\lambda>\omega$, as the
following system 
\begin{equation}  \label{3.11}
\left\lbrace 
\begin{array}{ll}
v_\lambda(t,t_0) & =(S_A \diamond
(Bv_\lambda(\cdot,t_0))(t_0+\cdot))(t-t_0)+\lambda R_\lambda(A)w(t,t_0), \
t\geq t_0\vspace{0.3cm} \\ 
v(t,t_0) & =(S_A \diamond (Bv(\cdot,t_0))(t_0+\cdot))(t-t_0)+ w(t,t_0), \
t\geq t_0%
\end{array}
\right.
\end{equation}
where we have used the notation $(Bv_\lambda
(\cdot,t_0))(t)=B(t)v_\lambda(t,t_0)$ for every $t\geq t_0$. \newline
Let $I\subset\mathbb{R}$ be a compact subset of $\mathbb{R}$. To show that (%
\ref{3.6}) exists uniformly for $t\geq t_0$ in $I$, we will make use of
Proposition \ref{PROP2.16}.\newline
We have from (\ref{3.11}) that for every $\lambda >\omega$ 
\begin{equation}  \label{3.12}
v_\lambda(t,t_0)-v(t,t_0)=(S_A \diamond
(B(v_\lambda(\cdot,t_0)-v(\cdot,t_0)))(t_0+\cdot))(t-t_0)+[\lambda
R_\lambda(A)-I]w(t), \ \forall t\geq t_0
\end{equation}
with the notation 
\begin{equation*}
(B(v_\lambda(\cdot,t_0)-v(\cdot,t_0)))(t):=B(t)(v_\lambda(t,t_0)-v(t,t_0)),
\ \forall t\geq t_0.
\end{equation*}
Let $\varepsilon>0$ be given such that 
\begin{equation}  \label{3.13}
2\varepsilon \ \underset{s\in I}{\sup}\ b(s)<\frac{1}{4}.
\end{equation}
Let $\tau_\varepsilon >0$ be given with $M\delta (\tau_\varepsilon)\leq
\varepsilon$. Then by using (\ref{3.12}) and Proposition \ref{PROP2.16} we
obtain for each $\lambda >\omega$ and each $t\geq t_0$ with $t,t_0\in I$ 
\begin{equation*}
\Vert v_\lambda(t,t_0)-v(t,t_0)\Vert \leq C(\varepsilon,\gamma) \underset{ 
\underset{s,t_0\in I}{s\geq t_0} }{\sup} \left[ e^{\gamma (t-s)}\ b(s)\
\Vert v_\lambda(s,t_0)-v(s,t_0)\Vert \right]+ \Vert [\lambda
R_\lambda(A)-I]w(t,t_0)\Vert,
\end{equation*}
whenever $\gamma \in (\omega,+\infty)$ with 
\begin{equation*}
C(\varepsilon,\gamma):=\dfrac{2 \varepsilon \max (1,e^{-\gamma
\tau_\varepsilon })}{1-e^{(\omega-\gamma) \tau_\varepsilon}},
\end{equation*}
so that 
\begin{equation*}
\begin{array}{ll}
\underset{ \underset{s,t_0\in I}{s\geq t_0} }{\sup} e^{-\gamma s} \Vert
v_\lambda(s,t_0)-v(s,t_0)\Vert \leq & C(\varepsilon,\gamma)\ \underset{s\in I%
}{\sup}\ b(s) \ \underset{ \underset{s,t_0\in I}{s\geq t_0} }{\sup}
e^{-\gamma s}\Vert v_\lambda(s,t_0)-v(s,t_0)\Vert \\ 
& +\underset{ \underset{s,t_0\in I}{s\geq t_0} }{\sup} \Vert [\lambda
R_\lambda(A)-I]w(s,t_0)\Vert.%
\end{array}%
\end{equation*}
By using (\ref{3.13}) one can chose $\gamma>0$ large enough such that 
\begin{equation*}
0\leq C(\varepsilon,\gamma)\ \underset{s\in I}{\sup}\ b(s)\ <\frac{1}{2},
\end{equation*}
providing for all $\lambda>\omega$ that 
\begin{equation*}
\underset{ \underset{s,t_0\in I}{s\geq t_0} }{\sup} e^{-\gamma s} \Vert
v_\lambda(s,t_0)-v(s,t_0)\Vert \leq 2 \underset{ \underset{s,t_0\in I}{s\geq
t_0} }{\sup} e^{-\gamma s} \Vert [\lambda R_\lambda(A)-I]w(s,t_0)\Vert.
\end{equation*}
Hence recalling that the limit $\underset{\lambda \rightarrow +\infty}{\lim}
\lambda R_\lambda(A) y=y$ is uniform on relatively compact sets of $X_0$ and
by observing that $w(\cdot,\cdot)$ maps $I\times I$ into a relatively
compact set of $X_0$ we obtain 
\begin{equation*}
\underset{\lambda \rightarrow +\infty}{\lim} \underset{ \underset{s,t_0\in I}%
{s\geq t_0} }{\sup} e^{-\gamma s} \Vert [\lambda
R_\lambda(A)-I]w(s,t_0)\Vert=0
\end{equation*}
that is 
\begin{equation*}
\underset{\lambda \rightarrow +\infty}{\lim} \underset{ \underset{s,t_0\in I}%
{s\geq t_0} }{\sup} e^{-\gamma s}\Vert v_\lambda(s,t_0)-v(s,t_0)\Vert=0
\end{equation*}
and since $I$ is bounded this implies 
\begin{equation*}
\underset{\lambda \rightarrow +\infty}{\lim} \underset{ \underset{s,t_0\in I}%
{s\geq t_0} }{\sup} \Vert v_\lambda(s,t_0)-v(s,t_0)\Vert=0.
\end{equation*}
The proof is complete.
\end{proof}

\section{A uniform convergence results}

\label{Sect-Uniform convergence}

Let $BUC(\mathbb{R},X)$ be the space of bounded and uniformly continuous
functions on $\mathbb{R}$. The next proposition gives a uniform
approximation result subject to $f$ belongs to an appropriate subspace of $%
BUC(\mathbb{R},X)$.

\begin{proposition}
\label{PROP4.1} Let Assumptions \ref{ASS1.1}, \ref{ASS1.2} and \ref{ASS1.3}
be satisfied. Assume in addition that 
\begin{equation*}
\sup_{t\in \mathbb{R}} b(t)<+\infty.
\end{equation*}
Let $f\in BUC(\mathbb{R},X)$ with relatively compact range. Then, for any
fixed $t_0>0$ the limit 
\begin{equation*}
\underset{\lambda \rightarrow +\infty}{\lim}\int_{t-t_0}^t U_B(t,s)\lambda
R_\lambda(A)f(s)ds,
\end{equation*}
exists uniformly for $t\in \mathbb{R}$.
\end{proposition}

\begin{proof}
Let $t_0>0$ be fixed. Recall that for each $\lambda>\omega$ we have 
\begin{equation*}
v_\lambda(t,t-t_0)=\int_{t-t_0}^t U_B(t,s)\lambda R_\lambda(A)f(s)ds, \
\forall t\in \mathbb{R}.
\end{equation*}
Thus by using similar arguments in the proof of \textit{ii)} in Lemma \ref%
{LE3.3} we have for each $t\in \mathbb{R}$, each $\lambda>\omega$ and $%
\mu>\omega$ 
\begin{equation*}
\underset{s\in [t-t_0,t]}{\sup} e^{-\gamma s} \Vert
v_\lambda(s,s-t_0)-v_\mu(s,s-t_0)\Vert \leq 2 \ \underset{s\in [t-t_0,t]}{%
\sup} e^{-\gamma s} \Vert [\lambda R_\lambda(A)-\mu R_\mu(A)]w(s,s-t_0)\Vert,
\end{equation*}
with $\gamma>\max(0,\omega)$ (large enough) and 
\begin{equation*}
w(t_1,t_2)=(S_A\diamond f(t_2+\cdot))(t_1-t_2), \ \forall (t_1,t_2)\in
\Delta.
\end{equation*}
Hence for each $t\in \mathbb{R}$, each $\lambda>\omega$ and $\mu >\omega$ 
\begin{equation*}
\begin{array}{ll}
\Vert v_\lambda(t,t-t_0)-v_\mu(t,t-t_0)\Vert & \leq 2 \ \underset{s\in
[t-t_0,t]}{\sup} e^{\gamma (t-s)} \Vert [\lambda R_\lambda(A)-\mu
R_\mu(A)]w(s,s-t_0)\Vert \\ 
& \leq 2\ e^{\gamma t_0} \ \underset{s\in [t-t_0,t]}{\sup} \Vert [\lambda
R_\lambda(A)-\mu R_\mu(A)]w(s,s-t_0)\Vert%
\end{array}%
\end{equation*}
Then to prove our proposition, it is sufficient to show that 
\begin{equation*}
\underset{\lambda,\mu \rightarrow +\infty}{\lim} \underset{s\in \mathbb{R}}{%
\sup} \Vert [\lambda R_\lambda(A)-\mu R_\mu(A)]w(s,s-t_0)\Vert=0.
\end{equation*}
This can be achieved by proving that $w(\cdot,\cdot-t_0)$ maps $\mathbb{R}$
in a relatively compact subset of $X_0$. To do so we will prove that for any 
$\varepsilon>0$, there exists a relatively compact set $K$ such that 
\begin{equation*}
w(t,t-t_0)\in K + c\varepsilon B_{X_0}(0,1), \ \forall t\in \mathbb{R},
\end{equation*}
for some constant $c>0$ and $B_{X_0}(0,1)$ the closed unit ball of $X_0$.%
\newline
Let $\varepsilon>0$ be given and fixed. Then since $f$ has its range in a
relatively compact subset of $X$, there exists $\eta=\frac{t_0}{n}>0$, with $%
n\in \mathbb{N}\setminus \{0\}$ and a function $g:\mathbb{R}\rightarrow X$
such that $g$ is constant on each interval $[k\eta, (k+1)\eta)$, $k\in 
\mathbb{Z}$. Moreover the range of $g$ is contained in a finite set $%
K_0\subset X$ and 
\begin{equation*}
\underset{t\in \mathbb{R}}{\sup}\Vert f(t)-g(t)\Vert\leq \varepsilon.
\end{equation*}
Note that $g$ can be written as 
\begin{equation*}
g(t)=\underset{k\in \mathbb{Z}}{\sum} x_k \mathbbm{1}_{[k\eta
,k\eta+\eta)}(t), \ \forall t\in \mathbb{R},
\end{equation*}
with $x_k\in K_0$ for all $k\in \mathbb{Z}$. Then by Lemma \ref{LE3.4} it is
easy to see that $t \rightarrow (S_A\ast g)(t)$ is differentiable on $%
[0,+\infty)$ and we can write 
\begin{equation*}
w(t,t-t_0)=(S_A\diamond g(t-t_0+\cdot))(t_0)+(S_A \diamond
(f-g)(t-t_0+\cdot))(t_0), \ \forall t\in \mathbb{R}.
\end{equation*}
Let $t\in \mathbb{R}$ be fixed. Note that one can write 
\begin{equation*}
t=k_0\eta + r, \ \text{ with } r \in [0,\eta) \ \text{and} \ k_0\in \mathbb{Z%
},
\end{equation*}
providing that (recalling $t_0=n\eta$) 
\begin{equation*}
\begin{array}{ll}
(S_A\diamond g(t-t_0+\cdot))(t_0) & =\displaystyle\dfrac{d}{dt} \int_0^{t_0}
S_A(t_0-s)g(t-t_0+s)ds \\ 
& =\displaystyle \dfrac{d}{dt} \int_{t-t_0}^{t} S_A(t-s)g(s)ds \\ 
& =\displaystyle \dfrac{d}{dt} \int_{(k_0-n)\eta+r}^{k_0\eta+r}
S_A(t-s)g(s)ds \\ 
& =\displaystyle\underset{i=0}{\overset{n-1}{\sum}}\displaystyle \dfrac{d}{dt%
} \int_{(k_0-i-1)\eta+r}^{(k_0-i)\eta+r} S_A(k_0\eta+r-s)g(s)ds \\ 
& =\displaystyle\underset{i=0}{\overset{n-1}{\sum}}\displaystyle \dfrac{d}{dt%
}\left[ \int_{(k_0-i-1)\eta+r}^{(k_0-i)\eta} S_A(k_0\eta+r-s)x_{k_0-i-1}ds
\right. \\ 
& \qquad \qquad\displaystyle \left. + \int_{(k_0-i)\eta}^{(k_0-i)\eta+r}
S_A(k_0\eta+r-s)x_{k_0-i}ds\right]%
\end{array}%
\end{equation*}
therefore we obtain 
\begin{equation*}
\begin{array}{ll}
(S_A\diamond g(t-t_0+\cdot))(t_0) & =\displaystyle\underset{i=0}{\overset{n-1%
}{\sum}}\displaystyle\left[S_A(i\eta+\eta)-S_A(i\eta+r)\right]x_{k_0-i-1} \\ 
& \qquad \qquad+\displaystyle \displaystyle\underset{i=0}{\overset{n-1}{\sum}%
}\left[S_A(i\eta+r)-S_A(i\eta)\right]x_{k_0-i} \\ 
& =\displaystyle\underset{i=1}{\overset{n}{\sum}}\displaystyle\left[%
S_A(i\eta)-S_A(i\eta-\eta+r)\right]x_{k_0-i} \\ 
& \qquad \qquad+\displaystyle \displaystyle\underset{i=0}{\overset{n-1}{\sum}%
}\left[S_A(i\eta+r)-S_A(i\eta)\right]x_{k_0-i} \\ 
& =[S_A(n\eta)-S_A(n\eta-\eta+r)]x_{k_0-n} \\ 
& \qquad \qquad+\displaystyle\underset{i=1}{\overset{n-1}{\sum}}%
\displaystyle \left[S_A(i\eta+r)-S_A(i\eta-\eta+r)\right]%
x_{k_0-i}+S_A(r)x_{k_0} \\ 
& = T_{A_0}(n\eta-\eta+r)S_A(\eta-r)x_{k_0-n} \\ 
& \qquad \qquad+\displaystyle\underset{i=1}{\overset{n-1}{\sum}}%
\displaystyle T_{A_0}(i\eta-\eta+r)S_A(\eta)x_{k_0-i}+S_A(r)x_{k_0}, \\ 
& 
\end{array}%
\end{equation*}
so that we can claim that $t\rightarrow (S_A\diamond g(t-t_0+\cdot))(t_0) $
has its range in 
\begin{equation*}
K=\left \lbrace \underset{k=0}{\overset{n}{\sum}} T_{A_0}(s_k)S_A(l_k)x_k :
0\leq s_k,l_k \leq t_0 \text{ and } x_k\in K_0, \ k=0,\dots,n \right \rbrace.
\end{equation*}
Then recalling that 
\begin{equation*}
(t,x)\in [0,+\infty)\times X\rightarrow S(t)x \ \ \text{and} \ \ (t,x)\in
[0,+\infty)\times X_0\rightarrow T(t)x
\end{equation*}
are continuous, $K$ is clearly compact. \newline
To complete the proof it remains to give an estimate of $z(\cdot,\cdot-t_0)$
with 
\begin{equation*}
z(t_1,t_2):=(S_A \diamond (f-g)(t_2+\cdot))(t_1-t_2), \ \forall (t_1,t_2)\in
\Delta.
\end{equation*}
By using Proposition \ref{PROP2.16} one obtains 
\begin{equation*}
\Vert z(t_1,t_2) \Vert \leq C(1,\gamma_0) \underset{t\in [0,t_1-t_2]}{\sup}
e^{\gamma_0 (t_1-t_2-t) }\Vert f(t_2+t)-g(t_2+t) \Vert, \ \forall
(t_1,t_2)\in \Delta.
\end{equation*}
with $\gamma_0>\max(0,\omega)$, $M\delta(\tau_1)\leq 1$ and 
\begin{equation*}
C(1,\gamma_0):=\dfrac{2 \max (1,e^{-\gamma_0 \tau_1 })}{1-e^{(\omega-%
\gamma_0) \tau_1}}.
\end{equation*}
Therefore 
\begin{equation*}
\begin{array}{ll}
\underset{(t_1,t_2)\in \Delta}{\sup} \Vert z(t_1,t_2) \Vert & \leq
C(1,\gamma_0) e^{\gamma_0 (t_1-t_2) } \underset{t\in \mathbb{R}}{\sup}\Vert
f(t)-g(t) \Vert \\ 
& \leq C(1,\gamma_0) e^{\gamma_0 (t_1-t_2) } \varepsilon,%
\end{array}%
\end{equation*}
that is 
\begin{equation*}
\begin{array}{ll}
\underset{t\in \mathbb{R}}{\sup} \ \Vert z(t,t-t_0) \Vert \leq C(1,\gamma_0)
e^{\gamma_0 t_0 } \varepsilon, & 
\end{array}%
\end{equation*}
and the result follows.
\end{proof}

\section{Exponential dichotomy}

\label{Sect-Exponential dichotomy} In this section we consider the complete
orbit of the Cauchy problem (\ref{1.1}). Namely we consider a continuous map 
$u:(-\infty,+\infty)\rightarrow X_0$ as a mild solution of 
\begin{equation}
\dfrac{du(t)}{dt}=(A+B(t))u(t)+f(t),\text{ for }t \in \mathbb{R}.
\label{5.1}
\end{equation}%
This part is devoted to the proof of Theorems \ref{THEO1.10} and \ref%
{THEO1.11}. We will give necessary and sufficient condition for the
evolution family $\{U_B(t,s)\}_{(t,s)\in \Delta}\subset \mathcal{L}(X_0)$ to
have an exponential dichotomy. More precisely we will prove that the
existence of exponential dichotomy for $\{U_B(t,s)\}_{(t,s)\in \Delta}$ is
equivalent to the existence of integrated solution $u\in C(\mathbb{R},X_0)$
for all $f$ in an appropriate subspace of $C(\mathbb{R},X)$. \newline
In what follows when $\lbrace U(t,s)\rbrace_{(t,s)\in \Delta}\subset 
\mathcal{L}(Z)$ has an exponential dichotomy we define its associate Green's
operator function by 
\begin{equation*}
\Gamma(t,s):= \left\lbrace 
\begin{array}{ll}
U^+(t,s), \ \ \ \text{ if } t\geq s, &  \\ 
-U^-(s,t), \ \text{ if } t<s. & 
\end{array}
\right.
\end{equation*}

\begin{remark}
It is well known that when $\lbrace U(t,s)\rbrace_{(t,s)\in \Delta}\subset 
\mathcal{L}(Z)$ has an exponential dichotomy then for each $x\in Z$, the map 
$(t,s) \in \mathbb{R}^2 \rightarrow U^-(t,s)x$ is continuous from $\mathbb{R}%
^2$ into $Z$ (see \cite[Lemma VI.9.15]{Schnaubelt} or \cite[Lemma 9.17]%
{Nagel}).
\end{remark}

\begin{remark}
\label{RE5.2} It is easy to obtain from condition i) in Definition \ref%
{DEF1.8} 
\begin{equation*}
\Pi^+(t)\Pi^-(t)=\Pi^-(t)\Pi^+(t)=0_{\mathcal{L}(Z)}.
\end{equation*}
We also trivially have 
\begin{equation*}
U^+(t,t)=\Pi^{+}(t) \quad \text{and} \quad U^+(t,r)U^+(r,l)=U^+(t,l), \quad
\forall t\geq r\geq l,
\end{equation*}
while 
\begin{equation*}
U^-(t,t)=\Pi^{-}(t) \quad \text{and} \quad U^-(t,r)U^-(r,l)=U^-(t,l). \quad
\forall t, r,l\in \mathbb{R},
\end{equation*}
It follows that $U^+$ (respectively $U^-$) is a strongly continuous semiflow
(respectively flow). One may also observe that 
\begin{equation*}
\quad U^-(t,r)U(r,l)=U^-(t,l), \quad \forall (r,t),(r ,l)\in \Delta
\end{equation*}
and 
\begin{equation*}
\quad U^+(t,r)U(r,l)=U^+(t,l), \quad \forall (t,r),(r ,l)\in \Delta.
\end{equation*}
\end{remark}

\begin{notation}
Let $(Z,\Vert \cdot \Vert)$ be a Banach space. The following weighted Banach
spaces will be used in the sequel 
\begin{equation*}
BC^{\eta}(\mathbb{R},Z):=\left \lbrace f \in C(\mathbb{R},Z): \Vert f
\Vert_{\eta}:= \underset{t\in \mathbb{R}}{\sup }e^{-\eta \vert t \vert}\Vert
f(t) \Vert_Z<+\infty \right \rbrace, \quad \eta \geq 0.
\end{equation*}
Note that we have the following continuous embedding 
\begin{equation*}
BC^{\eta_1}(\mathbb{R},Z)\subseteq BC^{\eta_2}(\mathbb{R},Z) \quad \text{ if 
} \quad \eta_1\leq \eta_2.
\end{equation*}
If $\eta=0 $ we set 
\begin{equation*}
BC(\mathbb{R},Z):=\left \lbrace f \in C(\mathbb{R},Z): \Vert f
\Vert_{\infty}:= \underset{t\in \mathbb{R}}{\sup }\Vert f(t) \Vert_Z<+\infty
\right \rbrace, \quad \eta \geq 0.
\end{equation*}
and we define 
\begin{equation*}
C_{0}(\mathbb{R},Z):=\left \lbrace f \in BC(\mathbb{R},Z): \underset{%
t\rightarrow \pm \infty }{\lim } f(t)=0 \right \rbrace.
\end{equation*}
\end{notation}

The following result is well known in the context of exponential dichotomy.
We refer for instance to \cite{Baskakov, Latushkin, Levitan}.

\begin{theorem}
\label{TH5.4} Let $Z$ be a Banach space. Let $\lbrace
U(t,s)\rbrace_{(t,s)\in \Delta}\subset \mathcal{L}(Z)$ be an exponentially
bounded evolution family. Consider the following integral equation 
\begin{equation}  \label{5.1bis}
u(t)=U(t,t_0)u(t_0)+\int_{t_0}^t U(t,s)f(s)ds, \ (t,t_0)\in \Delta.
\end{equation}
Then the following properties are equivalent

\begin{itemize}
\item[i)] $\lbrace U(t,s)\rbrace_{(t,s)\in \Delta}\subset \mathcal{L}(Z)$
has an exponential dichotomy.

\item[ii)] Let $\mathcal{F} (\mathbb{R},Z)$ be the space $BC(\mathbb{R},Z)$
or $C_0(\mathbb{R},Z)$. Then for any $f\in \mathcal{F}(\mathbb{R},Z)$ there
exists a unique solution $u\in \mathcal{F}(\mathbb{R},Z)$ of (\ref{5.1bis}).
\end{itemize}

Moreover if $\lbrace U(t,s)\rbrace_{(t,s)\in \Delta}$ has an exponential
dichotomy then for each $f\in \mathcal{F}(\mathbb{R},Z)$ the unique solution 
$u\in \mathcal{F}(\mathbb{R},Z)$ of (\ref{5.1bis}) is given by 
\begin{equation*}
u(t)=\int_{-\infty}^{+\infty} \Gamma(t,s)f(s)ds, \ \forall t\in \mathbb{R},
\end{equation*}
where $\lbrace \Gamma(t,s)\rbrace_{(t,s)\in \mathbb{R}^2}\subset \mathcal{L}%
(Z)$ is the Green's operator function associated to $\lbrace
U(t,s)\rbrace_{(t,s)\in \Delta}.$
\end{theorem}

In what follow we will give an analogue of Theorem \ref{TH5.4} for the
evolution family $\lbrace U_B(t,s)\rbrace_{(t,s)\in \Delta}\subset \mathcal{L%
}(X_0)$. To do so we will first prove some estimates.

\begin{proposition}
\label{PROP5.4} Let Assumptions \ref{ASS1.1}, \ref{ASS1.2} and \ref{ASS1.3}
be satisfied. Assume in addition that 
\begin{equation*}
\sup_{t\in \mathbb{R}} b(t) <+\infty.
\end{equation*}
Then there exists a non decreasing function $\delta^*:
[0,+\infty)\rightarrow [0,+\infty)$ with $\delta^*(t)\rightarrow 0$ as $%
t\rightarrow0^+$ such that for each $f\in C(\mathbb{R},X)$ and $\lambda> w+1$
the map 
\begin{equation*}
v_\lambda(t,t_0)=\int_{t_0}^t U_B(t,s)\lambda R_\lambda(A)f(s)ds, \
(t,t_0)\in \Delta,
\end{equation*}
satisfies 
\begin{equation*}
\Vert v_\lambda (t,t_0)\Vert \leq \delta^*(t-t_0) \underset{s \in [t_0,t]}{%
\sup} \Vert f(s) \Vert, \ \forall (t,t_0)\in \Delta.
\end{equation*}
\end{proposition}

\begin{proof}
Let $\lambda> \omega$ be given. Thus by Lemma \ref{LE3.3} there exists $%
\gamma>\max(0,w)$ large enough (independent of $t_0$) such that for each $%
t\geq t_0$ 
\begin{equation*}
\underset{s\in [t_0,t]}{\sup} e^{-\gamma s} \Vert v_\lambda(s,t_0)\Vert \leq
2 \underset{s\in [t_0,t]}{\sup} e^{-\gamma s} \Vert \lambda
R_\lambda(A)w(s,t_0)\Vert,
\end{equation*}
with 
\begin{equation*}
w(t_1,t_2)=(S_A\diamond f(t_2+\cdot))(t_1-t_2), \ \forall (t_1,t_2)\in
\Delta.
\end{equation*}
Since $w(t_1,t_2)\in X_0$ for all $(t_1,t_2)\in \Delta$ and by Assumption %
\ref{ASS1.3} 
\begin{equation*}
\Vert w(t_1,t_2) \Vert \leq \delta(t_2-t_1) \underset{s\in [t_1,t_2]}{\sup}
\Vert f(s) \Vert, \ \forall (t_1,t_2)\in \Delta,
\end{equation*}
it follows that for each $\lambda>\omega$ and $t\geq t_0$ 
\begin{equation*}
\begin{array}{ll}
\underset{s\in [t_0,t]}{\sup} e^{-\gamma s} \Vert v_\lambda(s,t_0)\Vert & 
\leq 2 \underset{s\in [t_0,t]}{\sup} e^{-\gamma s} \Vert \lambda
R_\lambda(A_0)w(s,t_0)\Vert \\ 
& \leq 2 \dfrac{M \vert \lambda \vert}{\lambda-\omega}\underset{s\in [t_0,t]}%
{\sup} e^{-\gamma s} \Vert w(s,t_0)\Vert \\ 
& \leq 2 \dfrac{M \vert \lambda \vert}{\lambda-\omega}\underset{s\in [t_0,t]}%
{\sup} \left[e^{-\gamma s} \delta(s-t_0) \underset{l\in [t_0,s]}{\sup} \Vert
f(l)\Vert\right].%
\end{array}%
\end{equation*}
Then by using the fact that $\delta$ is non decreasing and $\gamma>0$ we
obtain for each $\lambda>\omega$ and $t\geq t_0$ 
\begin{equation*}
\underset{s\in [t_0,t]}{\sup} e^{-\gamma s} \Vert v_\lambda(s,t_0)\Vert \leq
2 \dfrac{M \vert \lambda \vert}{\lambda-\omega} e^{-\gamma t_0} \delta
(t-t_0)\underset{s\in [t_0,t]}{\sup} \Vert f(s)\Vert, \ \forall t\geq t_0,
\end{equation*}
providing that 
\begin{equation*}
\Vert v_\lambda(t,t_0)\Vert \leq 2 \dfrac{M \vert \lambda \vert}{%
\lambda-\omega} e^{\gamma (t-t_0)} \delta (t-t_0)\underset{s\in [t_0,t]}{\sup%
} \Vert f(s)\Vert, \ \forall t\geq t_0.
\end{equation*}
The proof is easily completed by using the fact that 
\begin{equation*}
\lambda> \omega+1 \Rightarrow \dfrac{\vert \lambda \vert}{\lambda-w} <
1+\vert \omega \vert.
\end{equation*}
\vspace{0.3cm}
\end{proof}

In the rest of this paper, the following assumption will be used.

\begin{assumption}
\label{ASS5.6} Assume that $\{ U_B(t,s) \}_{(t,s)\in \Delta}\subset \mathcal{%
L}( X_0) $ has an exponential dichotomy with exponent $\beta>0$, constant $%
\kappa\geq 1$ and strongly continuous projectors $\{ \Pi_B^+(t)\}_{t\in 
\mathbb{R}}\subset \mathcal{L}( X_0)$ and $\{ \Pi_B^-(t)\}_{t\in \mathbb{R}}
\subset \mathcal{L}( X_0)$.
\end{assumption}

Note that if $\{ U_B(t,s) \}_{(t,s)\in \Delta}$ has an exponential dichotomy
then combining Remark \ref{RE5.2} and condition \textit{iv)} in Definition %
\ref{DEF1.8} we have 
\begin{equation}  \label{5.2}
\underset{t\in \mathbb{R}}{\sup} \Vert \Pi_B^+(t)\Vert_{\mathcal{L}(Z)} \leq
\kappa \ \text{and } \ \underset{t\in \mathbb{R}}{\sup} \Vert
\Pi_B^-(t)\Vert_{\mathcal{L}(Z)}\leq \kappa.
\end{equation}

\begin{proposition}
\label{PROP5.7} Let Assumption \ref{ASS1.1} be satisfied. Let $\lbrace
U(t,s)\rbrace_{(t,s)\in \Delta}\subset \mathcal{L}(X_0)$ be a given
evolution family such that there exist $\widehat{M}\geq 1$, $\widehat{\omega}%
\in \mathbb{R}$ and 
\begin{equation*}
\Vert U(t,s)\Vert_{\mathcal{L}(X_0)}\leq \widehat{M} e^{\widehat{\omega}%
(t-s)}, \ \forall (t,s)\in \Delta.
\end{equation*}
Assume that for each $f\in C(\mathbb{R},X)$ the map 
\begin{equation*}
v_\lambda(t,t_0)=\int_{t_0}^t U(t,s)\lambda R_\lambda(A)f(s)ds, \ (t,t_0)\in
\Delta,
\end{equation*}
satisfies

\begin{equation*}
\Vert v_\lambda (t,t_0)\Vert \leq \delta^{**}(t-t_0) \underset{s \in [t_0,t]}%
{\sup} \Vert f(s) \Vert, \ \forall (t,t_0)\in \Delta,
\end{equation*}
with $\delta^{**}: [0,+\infty)\rightarrow [0,+\infty)$ a non decreasing
function such that $\delta^{**}(t)\rightarrow 0$ as $t\rightarrow0^+$.%
\newline
Let $\varepsilon>0$ be given and fixed. Then, for each $\tau_\varepsilon>0$
satisfying $\widehat{M} \delta^{**}(\tau_\varepsilon)\leq \varepsilon$ and
each $\lambda>\omega+1$ we have 
\begin{equation*}
\Vert v_\lambda(t,t_0) \Vert \leq \widetilde{C}(\varepsilon,\gamma,\widehat{%
\omega},\widehat{M}) \underset{s \in [t_0,t]}{\sup} e^{\gamma (t-s)} \Vert
f(s) \Vert, \ \forall (t,t_0)\in \Delta,
\end{equation*}
whenever $\gamma >\widehat{\omega}$ and $f\in C(\mathbb{R},X)$ with 
\begin{equation*}
\widetilde{C}(\varepsilon,\gamma,\widehat{\omega},\widehat{M}):=\widehat{M}
e^{\max(0,\widehat{\omega})\tau_\varepsilon} \dfrac{2 \varepsilon \max
(1,e^{-\gamma \tau_\varepsilon })}{1-e^{(\widehat{\omega}-\gamma)
\tau_\varepsilon}}.
\end{equation*}
\end{proposition}

\begin{proof}
Without loss of generality we can assume that $t_0=0$. 
%Note that for all  $t\geq t_0$ and $l\in [t_0,t]$ we have
%\begin{equation*}
%\begin{array}{ll}
%\Pi^+(t)v_\lambda(t,t_0)=\displaystyle U_B^+(t,l)\int_{t_0}^l U_B(l,s)\lambda R_\lambda(A)f(s)ds+\Pi^+(t)\int_{l}^t U_B(t,s)\lambda R_\lambda(A)f(s)ds\\
%\end{array}
%\end{equation*}
%so that 
%\begin{equation}\label{2.31}
%\Pi^+(t)v_\lambda(t,t_0)=U_B^+(t,l)v_\lambda(l,t_0)+\Pi^+(t)v_\lambda(t,l).
%\end{equation}
Let $\tau_\varepsilon>0$ be given such that 
\begin{equation*}
\widehat{M}\delta^{**}(s)\leq \varepsilon, \ \forall s\in
[0,\tau_\varepsilon].
\end{equation*}
Let $t\geq 0$ be fixed. Then since we can write $t=n\tau_\varepsilon+\theta$
with $\theta\in [0,\tau_\varepsilon)$ and $n\in \mathbb{N}$ we obtain 
\begin{equation*}
\begin{array}{ll}
v_\lambda(t,0) & =\displaystyle\int_{0}^t U(t,s)\lambda R_\lambda(A)f(s)ds
\\ 
& =\displaystyle\underset{k=0}{\overset{n-1}{\sum}}\int_{k\tau_%
\varepsilon}^{(k+1)\tau_\varepsilon} U(t,s)\lambda
R_\lambda(A)f(s)ds+\int_{n\tau_\varepsilon}^{t} U(t,s)\lambda
R_\lambda(A)f(s)ds \\ 
& =\displaystyle\underset{k=0}{\overset{n-1}{\sum}}U(t,(k+1)\tau_%
\varepsilon)\int_{k\tau_\varepsilon}^{(k+1)\tau_\varepsilon}
U((k+1)\tau_\varepsilon,s)\lambda R_\lambda(A)f(s)ds \\ 
& \displaystyle \qquad \qquad \qquad
+U(t,n\tau_\varepsilon)\int_{n\tau_\varepsilon}^{t}
U(n\tau_\varepsilon,s)\lambda R_\lambda(A)f(s)ds \\ 
& =\displaystyle\underset{k=0}{\overset{n-1}{\sum}}U(t,(k+1)\tau_%
\varepsilon)v_\lambda((k+1)\tau_\varepsilon,k\tau_\varepsilon)+U(t,n\tau_%
\varepsilon)v_\lambda(t,n\tau_\varepsilon), \\ 
& 
\end{array}%
\end{equation*}
so that 
\begin{equation}  \label{5.4}
v_\lambda(t,0)=\displaystyle U(t,n\tau_\varepsilon) \underset{k=0}{\overset{%
n-1}{\sum}}U(n\tau_\varepsilon,(k+1)\tau_\varepsilon)v_\lambda((k+1)\tau_%
\varepsilon,k\tau_\varepsilon)+U(t,n\tau_\varepsilon)v_\lambda(t,n\tau_%
\varepsilon).
\end{equation}
Next observe that for all $(r_0,r_1)\in \Delta$ and $r\geq r_0$ with $0\leq
r_0-r_1\leq \tau_\varepsilon$ we have 
\begin{equation}  \label{5.5}
\begin{array}{ll}
\Vert U(r,r_0)v_\lambda (r_0,r_1)\Vert & \leq \widehat{M} e^{\widehat{\omega}%
(r-r_0)} \Vert v_\lambda (r_0,r_1)\Vert \\ 
& \leq e^{\widehat{\omega}(r-r_0)}\widehat{M}\delta^*(r_0-r_1) \underset{s
\in [r_1,r_0]}{\sup} \Vert f(s) \Vert \\ 
& \leq e^{\widehat{\omega}(r-r_0)} \varepsilon \underset{s \in [r_1,r_0]}{%
\sup} \Vert f(s) \Vert.%
\end{array}%
\end{equation}
Let $\gamma>\widehat{\omega}$ be fixed. Set $\varepsilon_1:=\max(1,
e^{-\gamma \tau_\varepsilon})$. Let $k\in \mathbb{N}$ and $r\in
[k\tau_\varepsilon,(k+1)\tau_\varepsilon]$ be given and fixed. \newline
Then if $\gamma\geq 0$ we have 
\begin{equation}  \label{5.6}
\varepsilon \underset{s \in [k \tau_\varepsilon,r]}{\sup} \Vert f(s)
\Vert=\varepsilon \underset{s \in [k \tau_\varepsilon,r]}{\sup} e^{-\gamma
s} e^{\gamma s} \Vert f(s) \Vert \leq \varepsilon_1 e^{\gamma r} \underset{s
\in [k \tau_\varepsilon,r]}{\sup} e^{-\gamma s} \Vert f(s) \Vert,
\end{equation}
while if $\gamma<0$ 
\begin{equation*}
\begin{array}{lll}
\varepsilon \underset{s \in [k \tau_\varepsilon,r]}{\sup} \Vert f(s) \Vert & 
=\varepsilon \underset{s \in [k \tau_\varepsilon,r]}{\sup} e^{-\gamma s}
e^{\gamma s} \Vert f(s) \Vert &  \\ 
& \leq \varepsilon e^{\gamma k \tau_\varepsilon}\underset{s \in [k
\tau_\varepsilon,r]}{\sup} e^{-\gamma s} \Vert f(s) \Vert &  \\ 
& \leq \varepsilon e^{\gamma r} e^{-\gamma (r-k \tau_\varepsilon)}\underset{%
s \in [k \tau_\varepsilon,r]}{\sup} e^{-\gamma s} \Vert f(s) \Vert &  \\ 
& \leq \varepsilon e^{\gamma r} e^{-\gamma \tau_\varepsilon}\underset{s \in
[k \tau_\varepsilon,r]}{\sup} e^{-\gamma s} \Vert f(s) \Vert &  \\ 
& \leq \varepsilon_1 e^{\gamma r} \underset{s \in [k \tau_\varepsilon,r]}{%
\sup} e^{-\gamma s} \Vert f(s) \Vert. &  \\ 
&  & 
\end{array}%
\end{equation*}
Therefore for each $k\in \mathbb{N}$, each $r\in
[k\tau_\varepsilon,(k+1)\tau_\varepsilon]$ and $\gamma>\widehat{\omega}$ we
obtain 
\begin{equation}  \label{5.7}
\varepsilon \underset{s \in [k \tau_\varepsilon,r]}{\sup} \Vert f(s) \Vert
\leq \varepsilon_1 e^{\gamma r} \underset{s \in [k \tau_\varepsilon,r]}{\sup}
e^{-\gamma s} \Vert f(s) \Vert.
\end{equation}
By (\ref{5.5}) and (\ref{5.7}) we obtain for each $k\in \mathbb{N}$, each $%
r\geq (k+1)\tau_\varepsilon$ and $\gamma>\widehat{\omega}$ 
\begin{equation*}
\begin{array}{ll}
\Vert U(r,(k+1)\tau_\varepsilon)v_\lambda
((k+1)\tau_\varepsilon,k\tau_\varepsilon)\Vert \leq
e^{-(\beta+\gamma)(r-(k+1)\tau_\varepsilon)} \varepsilon_1 e^{\gamma r} 
\underset{s \in [k \tau_\varepsilon,(k+1)\tau_\varepsilon]}{\sup} e^{-\gamma
s} \Vert f(s) \Vert. & 
\end{array}%
\end{equation*}
\begin{equation}  \label{5.8}
\begin{array}{ll}
\Vert U(r,(k+1)\tau_\varepsilon)v_\lambda
((k+1)\tau_\varepsilon,k\tau_\varepsilon)\Vert \leq e^{(\widehat{\omega}%
-\gamma)(r-(k+1)\tau_\varepsilon)} \varepsilon_1 e^{\gamma r}\underset{s \in
[k \tau_\varepsilon,(k+1)\tau_\varepsilon]}{\sup} e^{-\gamma s} \Vert f(s)
\Vert. & 
\end{array}%
\end{equation}
Since $t-n\tau_\varepsilon\in [0,\tau_\varepsilon)$ we have from (\ref{5.5})
and (\ref{5.7}) that 
\begin{equation*}
\begin{array}{ll}
\Vert U(t,n\tau_\varepsilon)v_\lambda (t,n\tau_\varepsilon)\Vert & \leq e^{%
\widehat{\omega}(t-n\tau_\varepsilon)} \varepsilon \underset{s \in
[n\tau_\varepsilon,t]}{\sup} \Vert f(s) \Vert \\ 
& \leq e^{\widehat{\omega}(t-n\tau_\varepsilon)} \varepsilon_1 e^{\gamma t} 
\underset{s \in [n \tau_\varepsilon,t]}{\sup} e^{-\gamma s} \Vert f(s) \Vert,
\\ 
& \leq e^{\max(0,\widehat{\omega})\tau_\varepsilon} \varepsilon_1 e^{\gamma
t} \underset{s \in [n \tau_\varepsilon,t]}{\sup} e^{-\gamma s} \Vert f(s)
\Vert, \\ 
& 
\end{array}%
\end{equation*}
and by using (\ref{5.4}) and (\ref{5.8}) we obtain 
\begin{equation*}
\begin{array}{ll}
\Vert v_\lambda(t,0)\Vert & \leq \displaystyle \widehat{M} e^{\widehat{\omega%
} (t-n\tau_\varepsilon)} \underset{k=0}{\overset{n-1}{\sum}}\Vert
U(n\tau_\varepsilon,(k+1)\tau_\varepsilon)v_\lambda((k+1)\tau_\varepsilon,k%
\tau_\varepsilon)\Vert+\Vert U(t,n\tau_\varepsilon)v_\lambda
(t,n\tau_\varepsilon)\Vert \\ 
& \leq \ \widehat{M} e^{\widehat{\omega} (t-n\tau_\varepsilon)} %
\displaystyle \underset{k=0}{\overset{n-1}{\sum}} e^{(\widehat{\omega}%
-\gamma)(n\tau_\varepsilon-(k+1)\tau_\varepsilon)} \varepsilon_1 e^{\gamma
n\tau_\varepsilon}\underset{s \in [k \tau_\varepsilon,(k+1)\tau_\varepsilon]}%
{\sup} e^{-\gamma s} \Vert f(s) \Vert \\ 
& \qquad \qquad \qquad \qquad+ e^{\max(0,\widehat{\omega})\tau_\varepsilon}
\varepsilon_1 e^{\gamma t} \underset{s \in [n \tau_\varepsilon,t]}{\sup}
e^{-\gamma s} \Vert f(s) \Vert \\ 
&  \\ 
& \leq \ \widehat{M} e^{\widehat{\omega} (t-n\tau_\varepsilon)} e^{\gamma
n\tau_\varepsilon} \displaystyle \left[\underset{k=0}{\overset{n-1}{\sum}}
e^{(\widehat{\omega}-\gamma)(n-1-k))\tau_\varepsilon} \right]\varepsilon_1 
\underset{s \in [0,t]}{\sup} e^{-\gamma s} \Vert f(s) \Vert \\ 
& \qquad \qquad \qquad \qquad+ e^{\max(0,\widehat{\omega})\tau_\varepsilon}
\varepsilon_1 e^{\gamma t} \underset{s \in [n \tau_\varepsilon,t]}{\sup}
e^{-\gamma s} \Vert f(s) \Vert \\ 
& \leq \ \widehat{M} e^{(\widehat{\omega}-\gamma) (t-n\tau_\varepsilon)}
e^{\gamma t} \displaystyle \left[\underset{k=0}{\overset{n-1}{\sum}} e^{(%
\widehat{\omega}-\gamma)k\tau_\varepsilon} \right]\varepsilon_1 \underset{s
\in [0,t]}{\sup} e^{-\gamma s} \Vert f(s) \Vert \\ 
& \qquad \qquad \qquad \qquad+ e^{\max(0,\widehat{\omega})\tau_\varepsilon}
\varepsilon_1 e^{\gamma t} \underset{s \in [0,t]}{\sup} e^{-\gamma s} \Vert
f(s) \Vert \\ 
& 
\end{array}%
\end{equation*}
Then since $\widehat{\omega}-\gamma<0$ we obtain 
\begin{equation*}
\begin{array}{ll}
\Vert v_\lambda(t,0)\Vert & \leq \widehat{M} e^{\max(0,\widehat{\omega}%
)\tau_\varepsilon} e^{\gamma t}\displaystyle \left[1+\underset{k=0}{\overset{%
+\infty}{\sum}} (e^{(\widehat{\omega}-\gamma)\tau_\varepsilon})^k \right]%
\varepsilon_1 \underset{s \in [0,t]}{\sup} e^{-\gamma s} \Vert f(s) \Vert \\ 
& \leq \widehat{M} e^{\max(0,\widehat{\omega})\tau_\varepsilon} e^{\gamma t}%
\displaystyle \left[\dfrac{2}{1-e^{(\widehat{\omega}-\gamma)\tau_\varepsilon}%
} \right]\varepsilon_1 \underset{s \in [0,t]}{\sup} e^{-\gamma s} \Vert f(s)
\Vert.%
\end{array}%
\end{equation*}
The proof is complete.
\end{proof}

As a direct consequence of Propositions \ref{PROP5.7} and \ref{PROP5.4} we
obtain the following result.

\begin{proposition}
\label{PROP5.8} Let Assumptions \ref{ASS1.1}, \ref{ASS1.2}, \ref{ASS1.3} and %
\ref{ASS5.6} be satisfied. Assume in addition that 
\begin{equation*}
\sup_{t\in \mathbb{R}} b(t) <+\infty.
\end{equation*}
Let $\varepsilon>0$ be given and fixed. Then, for each $\tau_\varepsilon>0$
satisfying $\kappa \delta^*(\tau_\varepsilon)\leq \varepsilon$ and each $%
\lambda>\omega+1$ the map 
\begin{equation*}
v_\lambda(t,t_0)=\int_{t_0}^t U_B(t,s)\lambda R_\lambda(A)f(s)ds, \
(t,t_0)\in \Delta,
\end{equation*}
satisfies 
\begin{equation*}
\Vert \Pi^+(t) v_\lambda(t,t_0) \Vert \leq \widehat{C}(\varepsilon,\gamma) 
\underset{s \in [t_0,t]}{\sup} e^{\gamma (t-s)} \Vert f(s) \Vert, \ \forall
(t,t_0)\in \Delta,
\end{equation*}
whenever $\gamma >-\beta$ and $f\in C(\mathbb{R},X)$ with 
\begin{equation}  \label{5.9}
\widehat{C}(\varepsilon,\gamma):=\kappa \dfrac{2 \varepsilon \max
(1,e^{-\gamma \tau_\varepsilon })}{1-e^{-(\beta+\gamma) \tau_\varepsilon}}.
\end{equation}
\end{proposition}

\begin{proposition}
\label{PROP5.9} Let Assumptions \ref{ASS1.1}, \ref{ASS1.2}, \ref{ASS1.3} and %
\ref{ASS5.6} be satisfied. Assume in addition that 
\begin{equation*}
\sup_{t\in \mathbb{R}} b(t) <+\infty.
\end{equation*}
Let $\varepsilon>0$ be given and fixed. Then, for each $\tau_\varepsilon>0$
satisfying $\kappa \delta^*(\tau_\varepsilon)\leq \varepsilon$ and each $%
\lambda>\omega+1$ the map 
\begin{equation*}
v_\lambda(t,t_0)=\int_{t_0}^t U_B(t,s)\lambda R_\lambda(A)f(s)ds, \
(t,t_0)\in \Delta,
\end{equation*}
satisfies 
\begin{equation*}
\Vert U_B^-(t_0,t) v_\lambda(t,t_0) \Vert \leq \widehat{C}%
(\varepsilon,\gamma) \underset{s \in [t_0,t]}{\sup} e^{\gamma (s-t_0)} \Vert
f(s) \Vert, \ \forall (t,t_0)\in \Delta,
\end{equation*}
whenever $\gamma >-\beta$ and $f\in C(\mathbb{R},X)$ with $\widehat{C}%
(\varepsilon,\gamma)$ defined in (\ref{5.9}).
\end{proposition}

\begin{proof}
Let $(t,t_0)\in \Delta$ be given. Without loss of generality one can assume
that $t=0$. From now one fix $t_0\leq 0$. Let $\tau_\varepsilon>0$ be given
such that 
\begin{equation*}
\kappa\delta^*(s)\leq \varepsilon, \ \forall s\in [0,\tau_\varepsilon].
\end{equation*}
Then since we can write $t_0=-n\tau_\varepsilon-\theta$ with $\theta\in
[0,\tau_\varepsilon)$ and $n\in \mathbb{N}$ we obtain 
\begin{equation*}
\begin{array}{ll}
U_B^-(0,t_0)v_\lambda(0,t_0) & =\displaystyle\int_{t_0}^0 U_B^-(0,s)\lambda
R_\lambda(A)f(s)ds \\ 
& =\displaystyle\underset{k=0}{\overset{n-1}{\sum}}\int_{-(k+1)\tau_%
\varepsilon}^{-k\tau_\varepsilon} U_B^-(0,s)\lambda R_\lambda(A)f(s)ds \\ 
& \hspace{2cm}\displaystyle+\int_{t_0}^{-n\tau_\varepsilon}
U_B^-(0,s)\lambda R_\lambda(A)f(s)ds \\ 
& =\displaystyle\underset{k=0}{\overset{n-1}{\sum}}U_B^-(0,-k\tau_%
\varepsilon)\int_{-(k+1)\tau_\varepsilon}^{-k\tau_\varepsilon}
U_B^-(-k\tau_\varepsilon,s)\lambda R_\lambda(A)f(s)ds \\ 
& \qquad \qquad \qquad\displaystyle+U_B^-(0,-n\tau_\varepsilon)%
\int_{t_0}^{-n\tau_\varepsilon} U_B^-(-n\tau_\varepsilon,s)\lambda
R_\lambda(A)f(s)ds, \\ 
& 
\end{array}%
\end{equation*}
so that 
\begin{equation}  \label{5.10}
\Pi^-(0)v_\lambda(0,t_0)=\underset{k=0}{\overset{n-1}{\sum}}%
U_B^-(0,-k\tau_\varepsilon)v_\lambda(-k\tau_\varepsilon,-(k+1)\tau_%
\varepsilon)+U_B^-(0,-n\tau_\varepsilon)v_\lambda(-n\tau_\varepsilon,t_0).
\end{equation}
Since $U_B^-(0,t_0)$ is invertible from $\Pi^-(t_0)(X_0)$ into $%
\Pi^-(0)(X_0) $ with inverse $U_B^-(t_0,0)$, by applying $U_B^-(t_0,0)$ to (%
\ref{5.10}) we obtain 
\begin{equation*}
U_B^-(t_0,0)v_\lambda(0,t_0)=\underset{k=0}{\overset{n-1}{\sum}}%
U_B^-(t_0,-k\tau_\varepsilon)v_\lambda(-k\tau_\varepsilon,-(k+1)\tau_%
\varepsilon)+U_B^-(t_0,-n\tau_\varepsilon)v_\lambda(-n\tau_\varepsilon,t_0).
\end{equation*}
and by using the evolution property of $U_B^-$ it follows that 
\begin{equation}  \label{5.11}
\begin{array}{ll}
U_B^-(t_0,0)v_\lambda(0,t_0) & =\displaystyle U_B^-(t_0,-n\tau_\varepsilon)%
\underset{k=0}{\overset{n-1}{\sum}}U_B^-(-n\tau_\varepsilon,-k\tau_%
\varepsilon)v_\lambda(-k\tau_\varepsilon,-(k+1)\tau_\varepsilon) \\ 
& \hspace*{4cm}+U_B^-(t_0,-n\tau_\varepsilon)v_\lambda(-n\tau_%
\varepsilon,t_0).%
\end{array}%
\end{equation}
Next observe that for all $(r_0,r_1)\in \Delta$ and $r\leq r_1$ with $0\leq
r_0-r_1\leq \tau_\varepsilon$ we have 
\begin{equation}  \label{5.12}
\begin{array}{ll}
\Vert U^-_B(r,r_0)v_\lambda (r_0,r_1)\Vert & \leq \kappa e^{-\beta(r_0-r)}
\Vert v_\lambda (r_0,r_1)\Vert \\ 
& \leq e^{-\beta(r_0-r)} \kappa\delta^*(r_0-r_1) \underset{s \in [r_1,r_0]}{%
\sup} \Vert f(s) \Vert \\ 
& \leq e^{-\beta(r_0-r)} \varepsilon \underset{s \in [r_1,r_0]}{\sup} \Vert
f(s) \Vert.%
\end{array}%
\end{equation}
Let $\gamma>-\beta$ be fixed. Set $\varepsilon_1:=\max(1, e^{-\gamma
\tau_\varepsilon})$. Let $k\in \mathbb{N}$ and $r\in
[-(k+1)\tau_\varepsilon,-k\tau_\varepsilon]$ be given and fixed. \newline
Then if $\gamma\geq 0$ we have 
\begin{equation*}
\varepsilon \underset{s \in [r,-k\tau_\varepsilon]}{\sup} \Vert f(s)
\Vert=\varepsilon \underset{s \in [r,-k\tau_\varepsilon]}{\sup} e^{-\gamma
s} e^{\gamma s} \Vert f(s) \Vert \leq \varepsilon_1 e^{-\gamma r} \underset{%
s \in [r,-k\tau_\varepsilon]}{\sup} e^{\gamma s} \Vert f(s) \Vert,
\end{equation*}
while if $\gamma <0$ 
\begin{equation*}
\begin{array}{lll}
\varepsilon \underset{s \in [r,-k\tau_\varepsilon]}{\sup} \Vert f(s) \Vert & 
=\varepsilon \underset{s \in [r,-k\tau_\varepsilon]}{\sup} e^{-\gamma s}
e^{\gamma s} \Vert f(s) \Vert &  \\ 
& \leq \varepsilon e^{\gamma k \tau_\varepsilon}\underset{s \in[%
r,-k\tau_\varepsilon]}{\sup} e^{\gamma s} \Vert f(s) \Vert &  \\ 
& \leq \varepsilon e^{-\gamma r} e^{\gamma (r+k \tau_\varepsilon)} \underset{%
s \in [r,-k\tau_\varepsilon]}{\sup} e^{\gamma s} \Vert f(s) \Vert &  \\ 
& \leq \varepsilon e^{-\gamma r} e^{-\gamma \tau_\varepsilon}\underset{s \in
[r,-k\tau_\varepsilon]}{\sup} e^{\gamma s} \Vert f(s) \Vert &  \\ 
& \leq \varepsilon_1 e^{-\gamma r} \underset{s \in [r,-k\tau_\varepsilon]}{%
\sup} e^{\gamma s} \Vert f(s) \Vert. &  \\ 
&  & 
\end{array}%
\end{equation*}
Therefore for each $k\in \mathbb{N}$, each $r\in
[-(k+1)\tau_\varepsilon,-k\tau_\varepsilon]$ and $\gamma>-\beta$ we obtain 
\begin{equation}  \label{5.13}
\varepsilon \underset{s \in [r,-k\tau_\varepsilon]}{\sup} \Vert f(s)
\Vert\leq \varepsilon_1 e^{-\gamma r} \underset{s \in [r,-k\tau_\varepsilon]}%
{\sup} e^{\gamma s} \Vert f(s) \Vert.
\end{equation}
By (\ref{5.12}) and (\ref{5.13}) we obtain for each $k\in \mathbb{N}$, each $%
r\leq -(k+1)\tau_\varepsilon$ and $\gamma>-\beta$ we obtain 
\begin{equation*}
\begin{array}{ll}
\Vert U^-_B(r,-k\tau_\varepsilon )v_\lambda
(-k\tau_\varepsilon,-(k+1)\tau_\varepsilon)\Vert \leq
e^{(\beta+\gamma)(r+(k+1)\tau_\varepsilon)} e^{-\gamma r} \varepsilon_1%
\underset{s \in [-(k+1)\tau_\varepsilon,-k\tau_\varepsilon]}{\sup} e^{\gamma
s}\Vert f(s) \Vert. & 
\end{array}%
\end{equation*}
Since $-n\tau_\varepsilon-t_0=\theta\in [0,\tau_\varepsilon)$ we obtain from
(\ref{5.11}) and (\ref{5.13}) 
\begin{equation}  \label{5.14}
\begin{array}{ll}
\Vert U_B^-(t_0,-n\tau_\varepsilon)v_\lambda(-n\tau_\varepsilon,t_0) \Vert & 
\leq e^{\beta(t_0+n\tau_\varepsilon)} \varepsilon \underset{s \in
[t_0,-n\tau_\varepsilon]}{\sup} \Vert f(s) \Vert \\ 
& \leq e^{\beta(t_0+n\tau_\varepsilon)} \varepsilon_1 e^{-\gamma t_0} 
\underset{s \in [t_0,-n\tau_\varepsilon]}{\sup}e^{\gamma s} \Vert f(s) \Vert,%
\end{array}%
\end{equation}
and by using (\ref{5.11}) and (\ref{5.14}) it follows that 
\begin{equation*}
\begin{array}{ll}
\Vert U_B^-(t_0,0)\Pi^-(0)v_\lambda(0,t_0)\Vert & \leq \displaystyle \kappa
e^{\beta (t_0+n\tau_\varepsilon)}\left[ \underset{k=0}{\overset{n-1}{\sum}}
e^{(\beta+\gamma)(-n+k+1)\tau_\varepsilon} e^{\gamma n\tau_\varepsilon}
\varepsilon_1\underset{s \in [-(k+1)\tau_\varepsilon,-k\tau_\varepsilon]}{%
\sup} e^{\gamma s}\Vert f(s) \Vert \right] \\ 
& \hspace*{4cm}+ e^{\beta(t_0+n\tau_\varepsilon)} \varepsilon_1 e^{-\gamma
t_0} \underset{s \in [t_0,-n\tau_\varepsilon]}{\sup}e^{\gamma s} \Vert f(s)
\Vert \\ 
& \leq \displaystyle \kappa e^{\beta (t_0+n\tau_\varepsilon)}e^{\gamma
n\tau_\varepsilon}\left[ \underset{k=0}{\overset{n-1}{\sum}}
e^{(\beta+\gamma)(-n+k+1)\tau_\varepsilon} \varepsilon_1\underset{s \in
[t_0,0]}{\sup} e^{\gamma s}\Vert f(s) \Vert \right] \\ 
& \hspace*{4cm}+ e^{\beta(t_0+n\tau_\varepsilon)} \varepsilon_1 e^{-\gamma
t_0} \underset{s \in [t_0,0]}{\sup}e^{\gamma s} \Vert f(s) \Vert \\ 
& \leq \displaystyle \kappa e^{(\beta+\gamma)
(t_0+n\tau_\varepsilon)}e^{-\gamma t_0}\left[ \underset{k=-n+1}{\overset{0}{%
\sum}} (e^{(\beta+\gamma)\tau_\varepsilon})^k \varepsilon_1\underset{s \in
[t_0,0]}{\sup} e^{\gamma s}\Vert f(s) \Vert \right] \\ 
& \hspace*{4cm}+ e^{\beta(t_0+n\tau_\varepsilon)} \varepsilon_1 e^{-\gamma
t_0} \underset{s \in [t_0,0]}{\sup}e^{\gamma s} \Vert f(s) \Vert.%
\end{array}%
\end{equation*}
Finally since $\gamma+\beta>0$ and $t_0+n\tau_\varepsilon<0$ we get 
\begin{equation*}
\begin{array}{ll}
\Vert U_B^-(t_0,0)\Pi^-(0)v_\lambda(0,t_0)\Vert & \leq \displaystyle \kappa %
\left[ 1+\underset{k=-n+1}{\overset{0}{\sum}} (e^{(\beta+\gamma)\tau_%
\varepsilon})^k \right] \varepsilon_1 e^{-\gamma t_0} \underset{s \in [t_0,0]%
}{\sup}e^{\gamma s} \Vert f(s) \Vert \\ 
& \leq \displaystyle \kappa \left[ 1+\underset{k=-\infty}{\overset{0}{\sum}}
(e^{(\beta+\gamma)\tau_\varepsilon})^k \right] \varepsilon_1 e^{-\gamma t_0} 
\underset{s \in [t_0,0]}{\sup}e^{\gamma s} \Vert f(s) \Vert \\ 
& \leq \displaystyle \kappa \left[ \dfrac{2}{1-e^{-(\beta+\gamma)}} \right]
\varepsilon_1 e^{-\gamma t_0} \underset{s \in [t_0,0]}{\sup}e^{\gamma s}
\Vert f(s) \Vert. \\ 
& 
\end{array}%
\end{equation*}
\end{proof}

\begin{lemma}
\label{LE5.10} Let Assumptions \ref{ASS1.1}, \ref{ASS1.2}, \ref{ASS1.3} and %
\ref{ASS5.6} be satisfied. Assume in addition that 
\begin{equation*}
 \sup_{t\in \mathbb{R}} b(t) <+\infty.
\end{equation*}
Let $\eta \in [0,\beta)$ be given. Then for each $\lambda > \omega+1$, each $%
f\in BC^\eta(\mathbb{R},X)$ and $t\in \mathbb{R}$ 
\begin{equation}  \label{5.15}
\mathcal{J}^+_\lambda(f)(t):=\underset{t_0\rightarrow -\infty}{\lim}%
\int_{t_0}^{t} U_B^+(t,s)\lambda R_\lambda(A) f(s)ds:=\int_{-\infty}^{t}
U_B^+(t,s)\lambda R_\lambda(A) f(s)ds,
\end{equation}
exists. Moreover the following properties hold

\begin{itemize}
\item[i)] For each $\eta\in [0,\beta)$ and each $\lambda >\omega+1 $, $%
\mathcal{J}^+_\lambda $ is a bounded linear operator from $BC^\eta(\mathbb{R}%
,X)$ into itself. More precisely for any $\nu\in (-\beta,0)$ 
\begin{equation*}
\Vert \mathcal{J}^+_\lambda(f)\Vert_{\eta} \leq\widehat{C}(1,\nu)\Vert f
\Vert_{\eta}, \ \forall f\in BC^\eta(\mathbb{R},X)\ \text{ with } \ \eta\in
[0,-\nu],
\end{equation*}
where $\widehat{C}(1,\nu)$ is the constant introduced in Proposition \ref%
{PROP5.8}.

\item[ii)] For each $\eta\in [0,\beta)$, each $\lambda >\omega+1$ and each $%
f\in BC^\eta(\mathbb{R},X)$ we have 
\begin{equation}  \label{5.16}
\mathcal{J}^+_\lambda(f)(t)=\displaystyle U_B^+(t,l)\mathcal{J}%
^+_\lambda(f)(l)+\displaystyle \int_{l}^t U_B^+(t,s)\lambda R_\lambda(A)
f(s)ds, \ \forall (t,l)\in \Delta.
\end{equation}
\end{itemize}
\end{lemma}

\begin{proof}
Let $\eta\in [0,\beta) $ be given. Let $\lambda>\omega+1$ be given and
fixed. Recall 
\begin{equation*}
v_\lambda(t,t_0):=\int_{t_0}^{t} U_B(t,s)\lambda R_\lambda(A) f(s)ds, \
\forall (t,t_0)\in \Delta,
\end{equation*}
and observe that 
\begin{equation*}
\Pi^+(t)v_\lambda(t,t_0)=\int_{t_0}^{t} U_B^+(t,s)\lambda R_\lambda(A)
f(s)ds, \ \forall (t,t_0)\in \Delta,
\end{equation*}
and 
\begin{equation*}
\mathcal{J}^+_\lambda(f)(t)=\underset{t_0\rightarrow -\infty}{\lim}%
\Pi^+(t)v_\lambda(t,t_0), \ \forall t\in \mathbb{R}.
\end{equation*}
To prove the existence of the limit, we will show that for each fixed $t\in 
\mathbb{R}$, $\lbrace\Pi^+(t) v_\lambda(t,t_0)\rbrace_{t_0\leq t}$ is a
Cauchy sequence. Fix $t \in \mathbb{R}$. Let $f\in BC^\eta(\mathbb{R},X)$ be
given. Let $t_0,r\in \mathbb{R}$ such that $t_0\leq r\leq t$. Then we have 
\begin{equation*}
\begin{array}{ll}
\Pi^+(t)v_\lambda(t,t_0) & =\displaystyle U_B^+(t,r)\int_{t_0}^{r}
U_B^+(r,s)\lambda R_\lambda(A) f(s)ds+\int_{r}^{t} U_B^+(t,s)\lambda
R_\lambda(A) f(s)ds \\ 
& =U_B^+(t,r)v_\lambda(r,t_0)+\Pi^+(t)v_\lambda(t,r)%
\end{array}%
\end{equation*}
and 
\begin{equation}  \label{5.17}
\Pi^+(t)
v_\lambda(t,t_0)-\Pi^+(t)v_\lambda(t,r)=U_B(t,r)\Pi^+(r)v_\lambda(r,t_0).
\end{equation}
Hence by Proposition \ref{PROP5.8} we can find a constant $\widehat{C}%
(1,\gamma)>0$ with $\gamma \in (-\beta,-\eta)$ such that 
\begin{equation*}
\begin{array}{ll}
\Vert \Pi^+(t) v_\lambda(t,t_0)-\Pi^+(t)v_\lambda(t,r)\Vert & \leq \kappa
e^{-\beta (t-r)}\widehat{C}(1,\gamma) \underset{s \in [t_0,r]}{\sup}
e^{\gamma (r-s)} \Vert f(s) \Vert \\ 
& \leq \kappa e^{-\beta (t-r)}\widehat{C}(1,\gamma) \Vert f \Vert_\eta 
\underset{s \in [t_0,r]}{\sup} e^{\gamma (r-s)} e^{\eta \vert s\vert} \\ 
& \leq \kappa e^{-\beta (t-r)}\widehat{C}(1,\gamma) \Vert f\Vert_\eta
e^{-\gamma (t-r)}\underset{s \in [t_0,r]}{\sup} e^{\gamma (t-s)} e^{\eta
\vert s\vert} \\ 
& \leq \kappa e^{-(\beta+\gamma) (t-r)}\widehat{C}(1,\gamma) \Vert f
\Vert_\eta \underset{s \in [t_0,r]}{\sup} e^{\gamma (t-s)} e^{\eta (\vert
t-s\vert+ \vert t\vert)}. \\ 
& 
\end{array}%
\end{equation*}
Then using the fact that $\beta+\gamma>0$ and $\eta+\gamma<0$ we obtain 
\begin{equation*}
\begin{array}{ll}
\Vert \Pi^+(t) v_\lambda(t,t_0)-\Pi^+(t)v_\lambda(t,r)\Vert \leq \kappa
e^{-(\beta+\gamma) (t-r)}\widehat{C}(1,\gamma) \Vert f \Vert_\eta e^{\eta
\vert t\vert }, &  \\ 
& 
\end{array}%
\end{equation*}
that is 
\begin{equation*}
\underset{t_0,r\rightarrow -\infty}{\lim}\Vert \Pi^+(t)
v_\lambda(t,t_0)-\Pi^+(t)v_\lambda(t,r)\Vert=0.
\end{equation*}
This prove the existence of the limit (\ref{5.15}) for any fixed $t\in 
\mathbb{R}$.\newline
\textbf{Proof of \textit{i)}}: Let $\eta\in [0,\beta)$ be given. Let $\nu
\in (-\beta,0)$. By Proposition \ref{PROP5.8} we can find a constant $%
\widehat{C}(1,\nu)>0$ such that 
\begin{equation*}
\Vert \Pi^+(t)v_\lambda(t,t_0) \Vert \leq \widehat{C}(1,\nu) \underset{s \in
[t_0,t]}{\sup} e^{\nu (t-s)} \Vert f(s) \Vert, \ \forall (t,t_0)\in \Delta.
\end{equation*}
Then for all $(t,t_0)\in \Delta $ 
\begin{equation*}
\begin{array}{ll}
\Vert \Pi^+(t) v_\lambda(t,t_0)\Vert & \leq \widehat{C}(1,\nu) \underset{s
\in [t_0,t]}{\sup} e^{\nu(t-s)} \Vert f(s) \Vert \\ 
& \leq \widehat{C}(1,\nu) \Vert f\Vert_\eta \underset{s \in [t_0,t]}{\sup}
e^{\nu(t-s)} e^{\eta \vert s\vert} \\ 
& \leq \widehat{C}(1,\nu) \Vert f\Vert_\eta \underset{s \in [t_0,t]}{\sup}
e^{\nu(t-s)} e^{\eta \vert t- s\vert +\eta \vert t\vert} \\ 
& \leq \widehat{C}(1,\nu) \Vert f\Vert_\eta e^{\eta \vert t\vert}\underset{s
\in [t_0,t]}{\sup} e^{(\nu+\eta)(t-s)}, \\ 
& 
\end{array}%
\end{equation*}
and since $\nu+\eta<0$ we obtain 
\begin{equation}  \label{5.18}
\begin{array}{ll}
\Vert \Pi^+(t) v_\lambda(t,t_0)\Vert & \leq \widehat{C}(1,\nu) \Vert
f\Vert_\eta e^{\eta \vert t\vert}. \\ 
& 
\end{array}%
\end{equation}
The result follows by letting $t_0\rightarrow -\infty$ in (\ref{5.18}). 
\newline
\textbf{Proof of \textit{ii)}}: Let $\eta\in [0,\beta)$ and $(t,l)\in \Delta$
be given. Then 
\begin{equation*}
\begin{array}{ll}
\mathcal{J}^+_\lambda(f)(t) & =\displaystyle U_B^+(t,l)\int_{-\infty}^{l}
U_B^+(l,s)\lambda R_\lambda(A) f(s)ds+\displaystyle \int_{l}^t
U_B^+(t,s)\lambda R_\lambda(A) f(s)ds \\ 
& =\displaystyle U_B^+(t,l)\mathcal{J}^+_\lambda(f)(l)+\displaystyle %
\int_{l}^t U_B^+(t,s)\lambda R_\lambda(A) f(s)ds. \\ 
& 
\end{array}%
\end{equation*}
\end{proof}

\begin{lemma}
\label{LE5.11} Let Assumptions \ref{ASS1.1}, \ref{ASS1.2}, \ref{ASS1.3} and %
\ref{ASS5.6} be satisfied. Assume in addition that 
\begin{equation*}
\sup_{t\in \mathbb{R}} b(t) <+\infty.
\end{equation*}
Let $\eta \in [0,\beta)$ be given. Then for each $\lambda > \omega+1$, each $%
f\in BC^\eta(\mathbb{R},X)$ and $t_0\in \mathbb{R}$ 
\begin{equation}  \label{5.19}
\mathcal{J}^-_\lambda(f)(t_0):=-\underset{t\rightarrow +\infty}{\lim}%
\int_{t_0}^{t} U_B^-(t_0,s)\lambda R_\lambda(A)
f(s)ds:=-\int_{t_0}^{+\infty} U_B^-(t_0,s)\lambda R_\lambda(A) f(s)ds,
\end{equation}
exists. Moreover the following properties hold

\begin{itemize}
\item[i)] For each $\eta\in [0,\beta)$ and each $\lambda >\omega+1 $, $%
\mathcal{J}^-_\lambda $ is a bounded linear operator from $BC^\eta(\mathbb{R}%
,X)$ into itself. More precisely for any $\nu\in (-\beta,0)$ 
\begin{equation*}
\Vert \mathcal{J}^-_\lambda(f)\Vert_{\eta} \leq\widehat{C}(1,\nu)\Vert f
\Vert_{\eta}, \ \forall f\in BC^\eta(\mathbb{R},X)\ \text{ with } \ \eta\in
[0,-\nu],
\end{equation*}
where $\widehat{C}(1,\nu)$ is the constant introduced in Proposition \ref%
{PROP5.8}.

\item[ii)] For each $\eta\in [0,\beta)$, each $\lambda>\omega+1$ and each $%
f\in BC^\eta(\mathbb{R},X)$ we have 
\begin{equation}  \label{5.20}
\mathcal{J}^-_\lambda(f)(t)=\displaystyle U_B^-(t,l)\mathcal{J}%
^-_\lambda(f)(l)+\displaystyle \int_{l}^t U_B^-(t,s)\lambda R_\lambda(A)
f(s)ds, \ \forall (t,l)\in \Delta.
\end{equation}
\end{itemize}
\end{lemma}

\begin{proof}
Let $\eta \in \lbrack 0,\beta )$ be given. Let $\lambda >\omega +1$ be given
and fixed. Recall 
\begin{equation*}
v_{\lambda }(t,t_{0}):=\int_{t_{0}}^{t}U_{B}(t,s)\lambda R_{\lambda
}(A)f(s)ds,\ \forall (t,t_{0})\in \Delta .
\end{equation*}%
Observe that 
\begin{equation*}
U_{B}^{-}(t_{0},t)v_{\lambda
}(t,t_{0})=\int_{t_{0}}^{t}U_{B}^{-}(t_{0},s)\lambda R_{\lambda }(A)f(s)ds,\
\forall (t,t_{0})\in \Delta ,
\end{equation*}%
and 
\begin{equation*}
\mathcal{J}_{\lambda }^{-}(f)(t_{0})=-\underset{t\rightarrow +\infty }{\lim }%
z_{\lambda }(t,t_{0}),\ \forall t_{0}\in \mathbb{R},
\end{equation*}%
with 
\begin{equation}
z_{\lambda }(t,t_{0}):=U_{B}^{-}(t_{0},t)v_{\lambda }(t,t_{0}),\ \forall
(t,t_{0})\in \Delta .  \label{5.21}
\end{equation}%
To prove the existence of the limit, we will show that for each $t_{0}\in 
\mathbb{R}$, $\{w_{\lambda }(t,t_{0})\}_{t\geq t_{0}}$ is a Cauchy sequence.
Let $f\in BC^{\eta }(\mathbb{R},X)$ be given. Let $t,r\in \mathbb{R}$ such
that $t_{0}\leq r\leq t$. Then we have 
\begin{equation*}
\begin{array}{ll}
z_{\lambda }(t,t_{0}) & =\displaystyle\int_{t_{0}}^{r}U_{B}^{-}(t_{0},s)%
\lambda R_{\lambda
}(A)f(s)ds+U_{B}^{-}(t_{0},r)\int_{r}^{t}U_{B}^{-}(r,s)\lambda R_{\lambda
}(A)f(s)ds \\ 
& =z_{\lambda }(r,t_{0})+U_{B}^{-}(t_{0},r)z_{\lambda }(r,t),%
\end{array}%
\end{equation*}%
and 
\begin{equation}
z_{\lambda }(t,t_{0})-z_{\lambda }(r,t_{0})=U_{B}^{-}(t_{0},r)z_{\lambda
}(r,t).  \label{5.22}
\end{equation}%
Then by Proposition \ref{PROP5.9} and the definition of $z_{\lambda }$ in (%
\ref{5.21}) we can find a constant $\widehat{C}(1,\gamma )>0$ with $\gamma
\in (-\beta ,-\eta )$ such that 
\begin{equation*}
\begin{array}{ll}
\Vert z_{\lambda }(t,t_{0})-z_{\lambda }(r,t_{0})\Vert  & \leq \kappa
e^{-\beta (r-t_{0})}\widehat{C}(1,\gamma )\underset{s\in \lbrack t,r]}{\sup }%
e^{\gamma (s-t)}\Vert f(s)\Vert  \\ 
& \leq \kappa e^{-\beta (r-t_{0})}\widehat{C}(1,\gamma )\Vert f\Vert _{\eta
}e^{-\gamma (t-t_{0})}\underset{s\in \lbrack t,r]}{\sup }e^{\gamma
(s-t_{0})}e^{\eta |s|} \\ 
& \leq \kappa e^{-\beta (r-t_{0})}\widehat{C}(1,\gamma )\Vert f\Vert _{\eta
}e^{-\gamma (r-t_{0})}e^{-\gamma (t-r)}\underset{s\in \lbrack t,r]}{\sup }%
e^{\gamma (s-t_{0})}e^{\eta |s|} \\ 
& \leq \kappa e^{-(\beta +\gamma )(r-t_{0})}\widehat{C}(1,\gamma )\Vert
f\Vert _{\eta }e^{-\gamma (t-r)}\underset{s\in \lbrack t,r]}{\sup }e^{\gamma
(s-t_{0})}e^{\eta |s|} \\ 
& \leq \kappa e^{-(\beta +\gamma )(r-t_{0})}\widehat{C}(1,\gamma )\Vert
f\Vert _{\eta }\underset{s\in \lbrack t,r]}{\sup }e^{\gamma
(s-t_{0})}e^{\eta |s|} \\ 
& \leq \kappa e^{-(\beta +\gamma )(r-t_{0})}\widehat{C}(1,\gamma )\Vert
f\Vert _{\eta }\underset{s\in \lbrack t,r]}{\sup }e^{\gamma
(s-t_{0})}e^{\eta (|s-t_{0}|+|t_{0}|)} \\ 
& \leq \kappa e^{-(\beta +\gamma )(r-t_{0})}\widehat{C}(1,\gamma )\Vert
f\Vert _{\eta }\underset{s\in \lbrack t,r]}{\sup }e^{(\gamma +\eta
)(s-t_{0})}e^{\eta |t_{0}|}, \\ 
& 
\end{array}%
\end{equation*}%
and since $\beta +\gamma >0$, $\gamma +\eta <0$ we obtain 
\begin{equation*}
\Vert z_{\lambda }(t,t_{0})-z_{\lambda }(r,t_{0})\Vert \leq \kappa
e^{-(\beta +\gamma )(r-t_{0})}\widehat{C}(1,\gamma )\Vert f\Vert _{\eta
}e^{\eta |t_{0}|},
\end{equation*}%
which gives 
\begin{equation*}
\underset{t,r\rightarrow +\infty }{\lim }\Vert z_{\lambda
}(t,t_{0})-z_{\lambda }(r,t_{0})\Vert =0,
\end{equation*}%
and proves the existence of the limit (\ref{5.19}).\newline
\textbf{Proof of \textit{i)}}: Let $\eta \in \lbrack 0,\beta )$ be given.
Let $\nu \in (-\beta ,0)$. Note that by Proposition \ref{PROP5.9} we can
find a constant $\widehat{C}(1,\nu )>0$ such that 
\begin{equation*}
\Vert w_{\lambda }(t,t_{0})\Vert \leq \widehat{C}(1,\nu )\underset{s\in
\lbrack t_{0},t]}{\sup }e^{\gamma (s-t_{0})}\Vert f(s)\Vert ,\ \forall
(t,t_{0})\in \Delta .
\end{equation*}%
Then for all $(t,t_{0})\in \Delta $ 
\begin{equation*}
\begin{array}{ll}
\Vert w_{\lambda }(t,t_{0})\Vert  & \leq \widehat{C}(1,\nu )\Vert f\Vert
_{\eta }\underset{s\in \lbrack t_{0},t]}{\sup }e^{\nu (s-t_{0})}e^{\eta |s|}
\\ 
& \leq \widehat{C}(1,\nu )\Vert f\Vert _{\eta }\underset{s\in \lbrack
t_{0},t]}{\sup }e^{\nu (s-t_{0})}e^{\eta (|s-t_{0}|+|t_{0}|)} \\ 
& \leq \widehat{C}(1,\nu )\Vert f\Vert _{\eta }\underset{s\in \lbrack
t_{0},t]}{\sup }e^{(\nu +\eta )(s-t_{0})}e^{\eta |t_{0}|}, \\ 
& 
\end{array}%
\end{equation*}%
and since $\nu +\eta <0$ we obtain 
\begin{equation}
\Vert w_{\lambda }(t,t_{0})\Vert \leq \widehat{C}(1,\nu )\Vert f\Vert _{\eta
}e^{\eta |t_{0}|}.  \label{5.23}
\end{equation}%
The result follows by letting $t\rightarrow +\infty $ in (\ref{5.23}).%
\newline
\textbf{Proof of \textit{ii)}}: Let $\eta \in \lbrack 0,\beta )$ and $%
(t,l)\in \Delta $ be given. Then 
\begin{equation*}
\begin{array}{ll}
\mathcal{J}_{\lambda }^{-}(f)(l) & =-\displaystyle\int_{l}^{t}U_{B}^{-}(l,s)%
\lambda R_{\lambda }(A)f(s)ds-\int_{t}^{+\infty }U_{B}^{-}(l,s)\lambda
R_{\lambda }(A)f(s)ds \\ 
& =-U_{B}^{-}(l,t)\displaystyle\int_{l}^{t}U_{B}^{-}(t,s)\lambda R_{\lambda
}(A)f(s)ds-U_{B}^{-}(l,t)\int_{t}^{+\infty }U_{B}^{-}(t,s)\lambda R_{\lambda
}(A)f(s)ds \\ 
& 
\end{array}%
\end{equation*}%
because $U_{B}^{-}(l,t)$ is invertible from $\Pi ^{-}(t)(X_{0})$ into $\Pi
^{-}(l)(X_{0})$ with inverse $U_{B}^{-}(t,l)$ and $\mathcal{J}_{\lambda
}^{-}(f)(l)\in \Pi ^{-}(l)(X_{0})$ one gets 
\begin{equation*}
\begin{array}{ll}
U_{B}^{-}(t,l)\mathcal{J}_{\lambda }^{-}(f)(l) & =-\displaystyle%
\int_{l}^{t}U_{B}^{-}(t,s)\lambda R_{\lambda }(A)f(s)ds-\int_{t}^{+\infty
}U_{B}^{-}(t,s)\lambda R_{\lambda }(A)f(s)ds \\ 
& =-\displaystyle\int_{l}^{t}U_{B}^{-}(t,s)\lambda R_{\lambda }(A)f(s)ds+%
\mathcal{J}_{\lambda }^{-}(f)(t),%
\end{array}%
\end{equation*}%
and the result follows.
\end{proof}

\begin{lemma}
\label{LE5.12} Let Assumptions \ref{ASS1.1}, \ref{ASS1.2}, \ref{ASS1.3} and %
\ref{ASS5.6} be satisfied. Assume in addition that 
\begin{equation*}
\sup_{t\in \mathbb{R}} b(t) <+\infty.
\end{equation*}
Let $\eta \in [0,\beta)$ be given. For each $\lambda > \omega+1$ and each $%
f\in BC^\eta(\mathbb{R},X)$ define 
\begin{equation}  \label{5.24}
\mathcal{J}_\lambda(f)(t):=\mathcal{J}^+_\lambda(f)(t)+\mathcal{J}%
^-_\lambda(f)(t):= \int_{-\infty}^{+\infty} \Gamma_B(t,s)\lambda
R_\lambda(A) f(s)ds,\ \forall t\in \mathbb{R},
\end{equation}

Then the following properties hold

\begin{itemize}
\item[i)] For each $\eta \in \lbrack 0,\beta )$ and each $\lambda >\omega +1$%
, $\mathcal{J}_{\lambda }$ is a bounded linear operator from $BC^{\eta }(%
\mathbb{R},X)$ into itself. More precisely for any $\nu \in (-\beta ,0)$ 
\begin{equation}
\Vert \mathcal{J}_{\lambda }(f)\Vert _{\eta }\leq 2\ \widehat{C}(1,\nu
)\Vert f\Vert _{\eta },\ \forall f\in BC^{\eta }(\mathbb{R},X)\ \text{ with }%
\ \eta \in \lbrack 0,-\nu ],  \label{5.24bis}
\end{equation}%
where $\widehat{C}(1,\nu )$ is the constant introduced in Proposition \ref%
{PROP5.8}.

\item[ii)] For each $\eta\in [0,\beta)$, each $\lambda>\omega+1$ and each $%
f\in BC^\eta(\mathbb{R},X)$ we have 
\begin{equation}  \label{5.25}
\mathcal{J}_\lambda(f)(t)=\displaystyle U_B(t,l)\mathcal{J}_\lambda(f)(l)+%
\displaystyle \int_{l}^t U_B(t,s)\lambda R_\lambda(A) f(s)ds, \ \forall
(t,l)\in \Delta.
\end{equation}

\item[iii)] For each $\eta\in [0,\beta)$, each $f\in BC^\eta(\mathbb{R},X)$, 
$\mathcal{J}_\lambda(f)$ is uniformly convergent on compact subset of $%
\mathbb{R}$ as $\lambda\rightarrow +\infty$.

\item[iv)] For each $f\in BUC(\mathbb{R},X)\subset BC^0(\mathbb{R},X)$ with
relatively compact range, $\mathcal{J}_\lambda(f)$ is uniformly convergent
on $\mathbb{R}$ as $\lambda\rightarrow +\infty$.
\end{itemize}
\end{lemma}

\begin{proof}
The proof of \textit{i)} follows from Lemmas \ref{LE5.10} and \ref{LE5.11}. 
\newline
\textbf{Proof of \textit{ii)}} Let $\eta\in [0,\beta)$, $f\in BC^\eta(%
\mathbb{R},X)$ and $(t,l)\in \Delta$ be given. Since $\mathcal{J}%
^+_\lambda(f)(l)\in \Pi^+(l) $, $\mathcal{J}^-_\lambda(f)(l)\in \Pi^-(l) $
one gets from (\ref{5.16}) and (\ref{5.20}) 
\begin{equation}  \label{5.26}
\mathcal{J}^+_\lambda(f)(t)=\displaystyle U_B(t,l)\Pi^+(l)\mathcal{J}%
^+_\lambda(f)(l)+\displaystyle \Pi^+(t)\int_{l}^t U_B(t,s)\lambda
R_\lambda(A) f(s)ds.
\end{equation}
and 
\begin{equation}  \label{5.27}
\mathcal{J}^-_\lambda(f)(t)=\displaystyle U_B(t,l)\Pi^-(l)\mathcal{J}%
^-_\lambda(f)(l)+\displaystyle \Pi^-(t)\int_{l}^t U_B(t,s)\lambda
R_\lambda(A) f(s)ds.
\end{equation}
and the result follows by adding up (\ref{5.26}) and (\ref{5.27}) combined
with the fact that $\Pi^+(t)+\Pi^-(t)=\Pi^+(l)+\Pi^-(l)=I$.\newline
\textbf{Proof of \textit{iii)}} To do this we will prove the convergence for 
$\mathcal{J}^+_\lambda$ and $\mathcal{J}^-_\lambda$ as $\lambda$ goes to $%
+\infty$. Let $\eta\in [0,\beta)$ and $f\in BC^\eta(\mathbb{R},X)$ be given. 
\newline
Let $\varepsilon>0$ be given given and fixed. Let $r>0$ be large enough such
that 
\begin{equation}  \label{5.28}
2\kappa e^{-\beta r} \widehat{C}(1,\nu) \Vert f\Vert_\eta\leq \varepsilon.
\end{equation}
We first prove the convergence for $\mathcal{J}^+_\lambda$ as $\lambda$ goes
to $+\infty$. Indeed by using (\ref{5.26}) combined with the estimate in 
\textit{i)} we obtain for each $\lambda,\mu >\omega+1$, each $t\in \mathbb{R}
$ 
\begin{equation*}
\begin{array}{ll}
\Vert \mathcal{J}^+_\lambda(f)(t)-\mathcal{J}^+_\mu(f)(t)\Vert & \leq %
\displaystyle \kappa e^{-\beta r} \Vert \mathcal{J}^+_\lambda(f)(t-r)-%
\mathcal{J}^+_\mu(f)(t-r)\Vert \\ 
& \qquad \qquad \qquad+\displaystyle \kappa \Vert \int_{t-r}^t
U_B(t,s)[\lambda R_\lambda(A)- \mu R_\mu(A) f(s)ds \Vert \\ 
& \leq \displaystyle 2 \kappa e^{-\beta r} \widehat{C}(1,\nu) \Vert
f\Vert_\eta \\ 
& \qquad \qquad \qquad+\displaystyle \kappa \Vert \int_{t-r}^t
U_B(t,s)[\lambda R_\lambda(A)- \mu R_\mu(A) f(s)ds \Vert \\ 
& 
\end{array}%
\end{equation*}
and by using (\ref{5.28}) we obtain the estimate 
\begin{equation}  \label{5.29}
\begin{array}{ll}
\Vert \mathcal{J}^+_\lambda(f)(t)-\mathcal{J}^+_\mu(f)(t)\Vert \leq
\varepsilon+\displaystyle \kappa \Vert \int_{t-r}^t U_B(t,s)[\lambda
R_\lambda(A)- \mu R_\mu(A) f(s)ds \Vert, \ \forall t\in \mathbb{R}. & 
\end{array}%
\end{equation}
Now infer from Theorem \ref{TH1.6} that 
\begin{equation*}
\underset{\lambda,\mu\rightarrow +\infty}{\lim} \int_{t-r}^t
U_B(t,s)[\lambda R_\lambda(A)- \mu R_\mu(A) f(s)ds=0,
\end{equation*}
uniformly for $t$ in a compact subset of $\mathbb{R}$ and (\ref{5.29})
yields 
\begin{equation*}
\underset{\lambda,\mu\rightarrow +\infty}{\lim} \Vert \mathcal{J}%
^+_\lambda(f)(t)-\mathcal{J}^+_\mu(f)(t)\Vert \leq \varepsilon,
\end{equation*}
uniformly for $t$ in a compact subset of $\mathbb{R}$. Since $\varepsilon>0$
is arbitrary fixed we conclude by a Cauchy sequence argument that $\underset{%
\lambda\rightarrow +\infty}{\lim} \mathcal{J}^+_\lambda(f)(t)$ exists
uniformly for $t$ in a compact subset of $\mathbb{R}$.\newline
Now we prove the convergence for $\mathcal{J}^-_\lambda$. First recall that
for each $t\in \mathbb{R}$, $U^-_B(t+r,t)$ is invertible from $\Pi^-(t)(X_0)$
into $\Pi^-(t+r)(X_0)$ with inverse $U^-_B(t,t+r)$. Then by applying $%
U^-_B(t,t+r)$ to the left side of (\ref{5.27}) one gets for all $t\in 
\mathbb{R}$ 
\begin{equation*}
\begin{array}{ll}
U^-_B(t,t+r)\mathcal{J}^-_\lambda(f)(t+r) & =\displaystyle U^-_B(t,t+r)
U_B(t+r,t)\Pi^-(t)\mathcal{J}^-_\lambda(f)(t) \\ 
& +\displaystyle U^-_B(t,t+r) \Pi^-(t+r)\int_{t}^{t+r} U_B(t+r,s)\lambda
R_\lambda(A) f(s)ds, \ \forall t\in \mathbb{R},%
\end{array}%
\end{equation*}
that is 
\begin{equation*}
\begin{array}{ll}
U^-_B(t,t+r)\mathcal{J}^-_\lambda(f)(t+r) & =\displaystyle \mathcal{J}%
^-_\lambda(f)(t)+\displaystyle U^-_B(t,t+r)\int_{t}^{t+r} U_B(t+r,s)\lambda
R_\lambda(A) f(s)ds, \ \forall t\in \mathbb{R},%
\end{array}%
\end{equation*}
so that 
\begin{equation*}
\mathcal{J}^-_\lambda(f)(t) =U^-_B(t,t+r)\mathcal{J}^-_\lambda(f)(t+r)-%
\displaystyle U^-_B(t,t+r)\int_{t}^{t+r} U_B(t+r,s)\lambda R_\lambda(A)
f(s)ds, \ \forall t\in \mathbb{R}.
\end{equation*}
Then for each $\lambda,\mu > \omega+1$, each $t\in \mathbb{R}$ 
\begin{equation*}
\begin{array}{ll}
\Vert \mathcal{J}^-_\lambda(f)(t)-\mathcal{J}^-_\mu(f)(t) \Vert & \leq
\kappa e^{-\beta r}\Vert \mathcal{J}^-_\lambda(f)(t+r)-\mathcal{J}%
^-_\mu(f)(t+r)\Vert \\ 
& +\Vert \displaystyle \int_{t}^{t+r} U^-_B(t,s)[\lambda R_\lambda(A)- \mu
R_\mu(A)]f(s)ds\Vert \\ 
& \leq 2 \kappa e^{-\beta r} \widehat{C}(1,\nu) \Vert f\Vert_\eta \\ 
& +\kappa \Vert \displaystyle \int_{t}^{t+r} U_B(t+r,s)[\lambda
R_\lambda(A)- \mu R_\mu(A)]f(s)ds\Vert, \\ 
& 
\end{array}%
\end{equation*}
and by using (\ref{5.28}) we obtain the estimate 
\begin{equation}  \label{5.30}
\begin{array}{ll}
\Vert \mathcal{J}^-_\lambda(f)(t)-\mathcal{J}^-_\mu(f)(t) \Vert \leq
\varepsilon +\kappa \Vert \displaystyle \int_{t}^{t+r} U_B(t+r,s)[\lambda
R_\lambda(A)- \mu R_\mu(A)]f(s)ds\Vert. &  \\ 
& 
\end{array}%
\end{equation}
Now we infer from Theorem \ref{TH1.6} that 
\begin{equation*}
\underset{\lambda,\mu\rightarrow +\infty}{\lim} \int_{t}^{t+r}
U_B(t+r,s)[\lambda R_\lambda(A)- \mu R_\mu(A)] f(s)ds=0,
\end{equation*}
uniformly for $t$ in a compact subset of $\mathbb{R}$ and (\ref{5.30})
yields 
\begin{equation*}
\underset{\lambda,\mu\rightarrow +\infty}{\lim} \Vert \mathcal{J}%
^-_\lambda(f)(t)-\mathcal{J}^-_\mu(f)(t)\Vert \leq \varepsilon.
\end{equation*}
uniformly for $t$ in a compact subset of $\mathbb{R}$. Since $\varepsilon>0$
is arbitrary fixed we conclude by a Cauchy sequence argument that $\underset{%
\lambda\rightarrow +\infty}{\lim} \mathcal{J}^-_\lambda(f)(t)$ exists
uniformly for $t$ in a compact subset of $\mathbb{R}$.\newline
Finally we obtain that 
\begin{equation*}
\underset{\lambda\rightarrow +\infty}{\lim} \mathcal{J}_\lambda(f)(t)=%
\underset{\lambda\rightarrow +\infty}{\lim} \mathcal{J}^+_\lambda(f)(t)+%
\underset{\lambda\rightarrow +\infty}{\lim} \mathcal{J}^-_\lambda(f)(t),
\end{equation*}
exists uniformly for $t$ in a compact subset of $\mathbb{R}$.\newline
\textbf{Proof of \textit{iv)}} The proof use the same argument as in the
proof of \textit{iii)}. The uniform convergence on $\mathbb{R}$ is obtained
by using Proposition \ref{PROP4.1} which ensures that the limits 
\begin{equation*}
\underset{\lambda,\mu\rightarrow +\infty}{\lim} \int_{t-r}^t
U_B(t,s)[\lambda R_\lambda(A)- \mu R_\mu(A)]f(s)ds=0,
\end{equation*}
and 
\begin{equation*}
\underset{\lambda,\mu\rightarrow +\infty}{\lim} \int_{t}^{t+r}
U_B(t+r,s)[\lambda R_\lambda(A)- \mu R_\mu(A)]f(s)ds=0,
\end{equation*}
are uniform for $t\in \mathbb{R}$.
\end{proof}

Now we are ready to prove the analogue of Theorem \ref{TH5.4} for our
purpose.

\begin{theorem} \label{TH5.13}
Let Assumptions \ref{ASS1.1}, \ref{ASS1.2} and \ref{ASS1.3} be satisfied.
Assume in addition that 
\begin{equation*}
\sup_{t\in \mathbb{R}} b(t) <+\infty.
\end{equation*}
Then the following assertions are equivalent

\begin{itemize}
\item[i)] The evolution family $\{ U_B(t,s)\}_{(t,s)\in \Delta}$ has an
exponential dichotomy.

\item[ii)] For each $f\in BC(\mathbb{R},X)$, there exists a unique
integrated solution $u\in BC(\mathbb{R},X_0)$ of (\ref{1.1}).
\end{itemize}

Moreover if $U_B$ has an exponential dichotomy with exponent $\beta>0$, then
for each $\eta \in [0,\beta)$ and each $f\in BC^\eta(\mathbb{R},X)$ there
exists a unique integrated solution $u\in BC^\eta(\mathbb{R},X_0)$ of (\ref%
{1.1}) which is given by 
\begin{equation*}
u(t)=\underset{\lambda \rightarrow +\infty}{\lim}\mathcal{J}_\lambda(f)(t)=%
\underset{\lambda \rightarrow +\infty}{\lim} \int_{-\infty}^{+\infty}
\Gamma_B(t,s)\lambda R_\lambda(A) f(s)ds, \ \forall t\in \mathbb{R},
\end{equation*}
where $\{ \Gamma_B(t,s) \}_{(t,s)\in \mathbb{R}^2}$ is the Green's operator
function associated to $\{ U_B(t,s) \}_{(t,s)\in \Delta}$.
\end{theorem}

\begin{proof}
\textbf{\textit{i)}} $\Rightarrow$\textbf{\textit{ii)}} This is a direct
consequence of Lemma \ref{LE5.12} by taking the limit when $\lambda$ goes to 
$+\infty$ in (\ref{5.25}).\newline
\textbf{\textit{ii)}} $\Rightarrow$\textbf{\textit{i)}} First of all note
that since $BC(\mathbb{R},X_0)\subset BC(\mathbb{R},X)$, the property 
\textit{ii)} ensures that for each $f\in BC(\mathbb{R},X_0)$ there exists a
unique integrated solution $u\in BC(\mathbb{R},X_0)$ of (\ref{1.1}).
Furthermore note that if $u_f\in BC(\mathbb{R},X_0)$ is a solution of (\ref%
{1.1}) for $f\in BC(\mathbb{R},X_0)$ then by Corollary \ref{COR3.1} we know
that it satisfies the integral equation 
\begin{equation*}
u_f(t)=U_{B}(t,t_0)x_0+\int_{t_0}^t U_B(t,s)f(s)ds, \ \forall t\geq t_0,
\end{equation*}
and \textit{i)} follows by Theorem \ref{TH5.4}. The proof is complete.
\end{proof}

As a consequence of the foregoing theorem we can obtain the following
persistence result for exponential dichotomy

\begin{theorem}  \label{TH5.14}
Let Assumptions \ref{ASS1.1}, \ref{ASS1.2}, \ref{ASS1.3} and \ref{ASS5.6} be
satisfied and assume in addition that 
\begin{equation*}
\sup_{t\in \mathbb{R}} b(t) <+\infty.
\end{equation*}
Then there exists $\varepsilon>0$ such that for each strongly continuous
family $\{C(t)\}_{t \in \mathbb{R}} \subset \mathcal{L}(X_0,X)$ satisfying 
\begin{equation*}
\sup_{t\in \mathbb{R}} \Vert B(t)-C(t) \Vert_{\mathcal{L}(X_0,X)} \leq
\varepsilon,
\end{equation*}
the evolution family generated by 
\begin{equation}  \label{5.31}
\dfrac{du(t)}{dt}=(A+C(t))u(t),\text{ for }t \in \mathbb{R}.
\end{equation}
has an exponential dichotomy.
\end{theorem}

\begin{proof}
The proof of this theorem is classical. Then we will only sketch the proof.
Note that the the evolution family generated by (\ref{5.31}) has and
exponential dichotomy if and only if for each $f\in BC(\mathbb{R},X)$ there
exists a unique $u\in BC(\mathbb{R},X_{0})$ satisfying 
\begin{equation*}
\dfrac{du(t)}{dt}=(A+C(t))u(t)+f(t),\text{ for }t\in \mathbb{R}.
\end{equation*}%
or equivalently 
\begin{equation*}
\dfrac{du(t)}{dt}=(A+B(t))u(t)+[C(t)-B(t)]u(t)+f(t),\text{ for }t\in \mathbb{%
R}.
\end{equation*}%
This is equivalent to solve for each $f\in BC(\mathbb{R},X)$ the fixed point
problem to find $u\in BC(\mathbb{R},X_{0})$ 
\begin{equation*}
u(t)=\mathcal{J}(\left[ C(.)-B(.)\right] u(.)+f(.))(t)
\end{equation*}%
where 
\begin{equation*}
\mathcal{J}(g)(t)=\underset{\lambda \rightarrow +\infty }{\lim }%
\int_{-\infty }^{+\infty }\Gamma _{B}(t,s)\lambda R_{\lambda }(A)g(s)ds,\
\forall t\in \mathbb{R}
\end{equation*}%
which can be performed by using the uniform estimates (\ref{5.24bis}) (for $%
\eta =0$) obtained in Lemma \ref{LE5.12}.
\end{proof}

\section{Example}

In order to illustrate our results we will apply some of the result to parabolic equation. Let $p\in \lbrack 1,+\infty )$ and $I:=(0,1)$.  Consider the following parabolic equation with non local boundary condition for each initial time $t_0 \in \mathbb{R}$ 
\begin{equation}
\left\{ 
\begin{array}{l}
\dfrac{\partial u(t,x)}{\partial t}=\dfrac{\partial
^{2}u(t,x)}{\partial x^{2}}+\alpha u(t,x)+g(t,x),\text{ for }t\geq t_0
\text{ and }x\in (0,1), \vspace{0.2cm}\\ 
-\dfrac{\partial u(t,0)}{\partial x}=\displaystyle\int_{I}\beta
_{0}(t,x)\varphi (x)dx+h_{0}(t) \vspace{0.2cm}\\ 
+\dfrac{\partial u(t,1)}{\partial x}=\displaystyle\int_{I}\beta
_{1}(t,x)\varphi (x)dx+h_{1}(t) \vspace{0.2cm}\\ 
u(t_0,.)=\varphi \in L^{p}(I,\mathbb{R}),%
\end{array}%
\right.   \label{6.1}
\end{equation}%
with $\alpha >0$, $g\in C(\mathbb{R},L^{p}(I,\mathbb{R}))$,  $
h_{0},h_{1}\in C(\mathbb{R},\mathbb{R})$ and $\beta _{0},\beta _{1}\in C(\mathbb{R},L^{q}(I,\mathbb{R}))$ (with $\frac{1}{p}+\frac{1}{q}=1$).

\bigskip 

\noindent \textbf{Abstract reformulation:} In order to incorportate the
boundary condition into the state variable, we consider%
\begin{equation*}
X:=\mathbb{R}^{2}\mathbb{\times }L^{p}(I,\mathbb{R})
\end{equation*}%
which is a Banach space endowed with the usual product norm 
\begin{equation*}
\left\Vert \left( 
\begin{array}{c}
x_{0} \\ 
x_{1} \\ 
\varphi 
\end{array}%
\right) \right\Vert =\left\vert x_{0}\right\vert +\left\vert
x_{1}\right\vert +\left\Vert \varphi \right\Vert _{L^{p}}
\end{equation*}%
and we set 
\begin{equation*}
X_{0}:=\left\{ 0_{\mathbb{R}^{2}}\right\} \times L^{p}(I,\mathbb{R}).
\end{equation*}%
We consider $\mathcal{A}:D(\mathcal{A})\subset X\rightarrow X$ the linear operator defined by 
\begin{equation*}
\mathcal{A}\left( 
\begin{array}{c}
0_{\mathbb{R}^2} \\ 
\varphi 
\end{array}%
\right) :=\left( 
\begin{array}{c}
\varphi ^{\prime }(0) \\ 
-\varphi ^{\prime }(1) \\ 
\varphi ^{\prime \prime } 
\end{array}%
\right) 
\end{equation*}%
with 
\begin{equation*}
D(\mathcal{A}):=\left\{ 0_{\mathbb{R}^{2}}\right\} \times W^{2,p}(I,\mathbb{R}).
\end{equation*}
By construction $\mathcal{A}_{0}$ the part of $\mathcal{A}$ in $X_{0}$ coincides with the usual
formulation for the parabolic equation (\ref{6.1}) with homogeneous boundary conditions. Indeed
$\mathcal{A}_{0}:D(\mathcal{A}_{0})\subset
X_{0}\rightarrow X_{0}$ is a linear operator on $X_{0}$ defined by
\begin{equation*}
\mathcal{A}_{0}\left(
\begin{array}{c}
0_{\mathbb{R}^2} \\
\varphi%
\end{array}%
\right) =\left(
\begin{array}{c}
0_{\mathbb{R}^2} \\
 \varphi ^{\prime \prime }%
\end{array}%
\right)
\end{equation*}%
with
\begin{equation*}
D(\mathcal{A}_{0})=\{ 0_{\mathbb{R}^2} \} \times \left\{ \varphi \in W^{2,p}\left( \left(
0,1\right) ,\mathbb{R} \right) :\varphi ^{\prime }(0)=\varphi ^{\prime }(1)=0 \right\} .
\end{equation*}%
In the following lemma we will first summarize some classical properties for the linear operator $\mathcal{A}_0$.
\begin{lemma}
\label{LE6.1}The linear operator $\mathcal{A}_{0}$ is the infinitesimal generator of $\left\lbrace T_{\mathcal{A}_0}(t) \right\rbrace_{t\geq 0}$ an analytic semigroup of bounded linear operator on $X_{0}$. Moreover $ T_{\mathcal{A}_0}(t) $ is compact for each $t>0$ and $\left( 0,+\infty \right) \subset \rho
(\mathcal{A}_0)$. The spectrum of $\mathcal{A}_{0}$ is given by  
$$
\sigma(\mathcal{A}_0)=\left\lbrace -  (\pi k)^2:k \in \mathbb{N} \right\rbrace
$$
and each eigenvalue $\lambda_k:=- (\pi k)^2$ is associated to the eigenfunction 
$$
\psi_k(x):=sin(\pi kx).
$$
Moreover each eigenvalue $\lambda_k$ is simple and the projector on the generalized eigenspace associated to this eigenvalue is given by 
$$
\Pi_{k,0}\left(
\begin{array}{c}
0_{\mathbb{R}^2}\\
\varphi
\end{array}
\right)
:=
\left(
\begin{array}{c}
0_{\mathbb{R}^2}\\
\frac{\int_0^1\psi_k(r) \varphi(r)dr}{\int_0^1\psi_k(r)^2dr} \psi_k
\end{array}
\right).
$$
\end{lemma}
Set
\begin{equation*}
\Omega _{\omega }=\left\{ \lambda \in \mathbb{C}:\mathcal{R}e \left( \lambda
\right) >\omega \right\} ,\ \forall \omega \in \mathbb{R},
\end{equation*}%
define for $\lambda \in \mathbb{C},$%
\begin{equation*}
\Delta \left( \lambda \right) := \mu^2 (e^{\mu }- e^{-\mu }),
\end{equation*}%
where
\begin{equation*}
\mu :=\sqrt{\lambda }.
\end{equation*}
Next we compute explicitly the resolvent of $\mathcal{A}$. 
\begin{lemma} \label{LE6.2} The resolvent of $\mathcal{A}$ is given by 
For each $\omega _{\mathcal{A}}\geq 0,$ such that
\begin{equation*}
\Omega _{\omega _{\mathcal{A}}}\subset \left\{ \lambda \in \mathbb{C}:\Delta \left(
\lambda \right) \neq 0\right\} \subset \rho \left( \mathcal{A}\right) ,
\end{equation*}%
and for each $\lambda :=\mu ^{2}\in \Omega _{\omega _{\mathcal{A}}}$ we have
\begin{equation*}
\begin{array}{l}
(\lambda I-\mathcal{A})
\left(
\begin{array}{c}
0_{\mathbb{R}^2} \\
\varphi
\end{array}
\right) 
=\left(
\begin{array}{c}
y_{0} \\
y_{1} \\
f%
\end{array}%
\right)
\Leftrightarrow \\
\left(
\begin{array}{c}
0_{\mathbb{R}^2} \\
\varphi
\end{array}
\right) =(\lambda I-\mathcal{A})^{-1}\left(
\begin{array}{c}
y_{0} \\
y_{1} \\
f%
\end{array}%
\right) 
\Leftrightarrow \\

\varphi (x)=\dfrac{\Delta_1(x)}{\Delta \left( \lambda \right) } \dfrac{1}{\mu} y_{0} +\dfrac{\Delta_2(x)}{\Delta \left( \lambda \right) } \dfrac{1}{\mu} y_{1}+\displaystyle\dfrac{\Delta _{1}(x)}{\Delta \left( \lambda \right) }\frac{1}{2\mu }
\int_{0}^{1}e^{-\mu s}f(s)ds\\  \qquad \ \ \  + \displaystyle
\dfrac{\Delta _{2}(x)}{\Delta \left( \lambda
\right) }\frac{1}{2\mu }\int_{0}^{1}e^{-\mu (1-s)}f(s)ds+\frac{1}{2\mu }%
\int_{0}^{1}e^{-\mu \left\vert x-s\right\vert }f(s)ds%
\end{array}%
\end{equation*}%
where
\begin{equation*}
\Delta _{1}(x)= \mu^2\left[  e^{\mu (1-x)}+ e^{-\mu (1-x)}\right] \text{ and } \Delta _{2}(x)= \mu^2 \left[  e^{-\mu x}+  e^{\mu x}\right].
\end{equation*}%
\end{lemma}
\begin{proof}
In order to compute the resolvent\ we set
\begin{equation*}
u(x):=\frac{1}{2\mu }\int_{0}^{1}e^{-\mu \left\vert x-s\right\vert }f(s)ds=%
\frac{1}{2\mu }\int_{-\infty }^{+\infty }e^{-\mu \left\vert x-s\right\vert }%
\bar{f}(s)ds
\end{equation*}%
where $\bar{f}$ extend $f$ by $0$ on $\mathbb{R}\setminus \left[ 0,1\right]
. $ We have%
\begin{equation*}
u(x)=\frac{1}{2\mu }\left[ \int_{-\infty }^{x}e^{-\mu \left( x-s\right) }%
\bar{f}(s)ds+\int_{x}^{+\infty }e^{\mu \left( x-s\right) }\bar{f}(s)ds\right]
\end{equation*}%
so
\begin{equation*}
u^{\prime }(x)=-\frac{1}{2}\int_{-\infty }^{x}e^{-\mu (x-s)}\bar{f}(s)ds+%
\frac{1}{2}\int_{x}^{+\infty }e^{\mu (x-s)}\bar{f}(s)ds.
\end{equation*}%
We set%
\begin{equation*}
u(0)=\gamma _{0}:=\frac{1}{2\mu }\int_{0}^{1}e^{-\mu s}f(s)ds\text{ and }%
u(1)=\gamma _{1}:=\frac{1}{2\mu }\int_{0}^{1}e^{-\mu (1-s)}f(s)ds
\end{equation*}%
and we observe that
\begin{equation*}
u^{\prime }(0)=\mu \gamma _{0}\text{ and }u^{\prime }(1)=-\mu \gamma _{1}
\end{equation*}%
We set
\begin{equation*}
u_{1}(x):=e^{-\mu x}\text{ and }u_{2}(x):=e^{\mu x}\text{.}
\end{equation*}%
In order to solve the problem
\begin{equation*}
(\lambda I-\mathcal{A})\left(
\begin{array}{c}
0 \\
0 \\
\varphi%
\end{array}%
\right) =\left(
\begin{array}{c}
y_{0} \\
y_{1} \\
f%
\end{array}%
\right)
\end{equation*}%
we look for $\varphi $ under the form
\begin{equation*}
\varphi (x)=u(x)+z_{1}u_{1}(x)+z_{2}u_{2}(x),
\end{equation*}%
where $z_{1},z_{2}\in \mathbb{R}$.

We observe that to verify the boundary conditions 
\begin{equation*}
-\varphi^{\prime }(0)=y_0 \text{ and } \varphi^{\prime }(1)=y_1
\end{equation*}
is equivalent to 
\begin{equation*}
\left\lbrace
\begin{array}{ll}
 u^{\prime }(0)+z_{1}u_{1}^{\prime }(0)+z_{2}u_{2}^{\prime}(0)  =-y_0\\
 u^{\prime }(0)+z_{1}u_{1}^{\prime }(0)+z_{2}u_{2}^{\prime}(0)=y_1
\end{array}
\right.
\end{equation*}
so we must solve the system
\begin{eqnarray*}
z_{1} u_{1}^{\prime }(0)+z_{2} u_{2}^{\prime }(0)&=&-y_{0}-u^{\prime }(0) \\
z_{1}u_{1}^{\prime }(1)+z_{2}u_{2}^{\prime }(1)&=&y_{1}-u^{\prime }(1)
\end{eqnarray*}%
which is equivalent to%
\begin{equation*}
\begin{array}{lll}
 -\mu  z_{1}&+ \mu z_{2} &=-y_{0}-\mu \gamma _{0}\\
-\mu e^{-\mu }z_{1}&+\mu e^{\mu }z_{2} &=y_{1}+\mu \gamma_{1}
\end{array}%
\end{equation*}%
hence
\begin{equation*}
\begin{array}{l}
z_{1}=\frac{1}{\Delta \left( \lambda \right)}\left[ - \mu e^{\mu }\left( -y_{0}-\mu \gamma _{0} \right) +  \mu \left( y_{1}+\mu \gamma_{1}\right) \right] \\
z_{2}=\frac{1}{\Delta \left( \lambda \right) } \left[ - \mu e^{-\mu }\left( -y_{0}-\mu \gamma _{0} \right) +  \mu \left( y_{1}+\mu \gamma_{1} \right) \right]
\end{array}
\end{equation*}
and the result follows.
\end{proof}

The following estimation shows that $\mathcal{A}$ is not Hille-Yosida. 
\begin{lemma}\label{LE6.3} We have the following estimations 
\begin{equation*}
0<\liminf_{\lambda \left( \in \mathbb{R}\right) \rightarrow +\infty }\lambda ^{%
\frac{1}{p^{\ast }}}\left\Vert \left( \lambda I-\mathcal{A}\right) ^{-1}\right\Vert _{%
\mathcal{L}(X)} \leq \limsup_{\lambda \left( \in \mathbb{R}\right) \rightarrow +\infty }\lambda ^{%
\frac{1}{p^{\ast }}}\left\Vert \left( \lambda I-\mathcal{A}\right) ^{-1}\right\Vert _{%
\mathcal{L}(X)}<+\infty ,
\end{equation*}%
with 
$$
p^{\ast }=\dfrac{2p}{1+p}.
$$
\end{lemma}
\begin{proof} Let $\lambda>0$ be  large enough. We have 
$$
\left\Vert (\lambda I-\mathcal{A})^{-1}\left(
\begin{array}{c}
0 \\
y_{1} \\
0_{L^p}
\end{array}%
\right) \right\Vert=\vert y_1 \vert  \dfrac{\sqrt{\lambda } }{\Delta \left( \lambda \right) } \left\Vert  e^{\sqrt{\lambda } \cdot}+e^{-\sqrt{\lambda } \cdot}\right\Vert_{L^p}
$$
Set 
$$
\gamma_{\lambda }:= \dfrac{\sqrt{\lambda } }{\Delta \left( \lambda \right) } \left\Vert  e^{\sqrt{\lambda } \cdot}+e^{-\sqrt{\lambda } \cdot}\right\Vert_{L^p}
$$
we have 
$$
\dfrac{\sqrt{\lambda } }{\Delta \left( \lambda \right) } \left[ \Vert  e^{\sqrt{\lambda } \cdot}\Vert_{L^p}- \Vert e^{-\sqrt{\lambda } \cdot}\Vert_{L^p} \right] \leq \gamma_{\lambda } \leq \dfrac{\sqrt{\lambda } }{\Delta \left( \lambda \right) } \left[ \Vert  e^{\sqrt{\lambda } \cdot}\Vert_{L^p}+ \Vert e^{-\sqrt{\lambda } \cdot}\Vert_{L^p} \right]
$$
and 
$$
\dfrac{\sqrt{\lambda } }{\Delta \left( \lambda \right) } \Vert  e^{\sqrt{\lambda } \cdot}\Vert_{L^p}=\dfrac{\sqrt{\lambda } }{ \lambda (e^{\sqrt{\lambda } }- e^{-\sqrt{\lambda } })  }
\left( \int_0^1 e^{p\sqrt{\lambda } x} dx\right)^{1/p}  
$$
and 
$$
\lim_{\lambda  \rightarrow +\infty} \gamma_{\lambda } \lambda^{\frac{p+1}{2p}}=(1/p)^{1/p}>0,
$$
and the result follows. 
\end{proof}

By using Lemmas \ref{LE6.1}-\ref{LE6.3}, we deduce that Assumption 3.4 in Ducrot, Magal and Prevost \cite{DMP} is satisfied. Therefore by applying Theorem 3.11 in \cite{DMP} we obtain the following lemma. 
\begin{lemma} \label{LE6.4}
The linear operator $\mathcal{A}$ satisfies Assumption \ref{ASS1.1} and Assumption \ref{ASS1.2}.  
\end{lemma}
\begin{remark}\label{REM6.5} Since $\rho(\mathcal{A}) \neq \emptyset$, one can prove that $\sigma(\mathcal{A}_0)=\sigma(\mathcal{A})$ (see \cite{Magal-Ruan2009b}).
\end{remark}

\noindent \textbf{Abstract Cauchy problem:} By identifying $u(t,.)$ and $v(t):=\left( 
\begin{array}{c}
0_{\mathbb{R}^{2}} \\ 
u(t,.)%
\end{array}%
\right) $ we can rewrite equation (\ref{6.1}) as the following abstract
Cauchy problem  for each initial time $t_0 \in \mathbb{R}$ 
\begin{equation}
\dfrac{dv(t)}{dt}=\mathcal{A} v(t)+\alpha v(t) +\mathcal{B}(t)v(t)+f(t),\text{ for }t\geq t_0\text{ and }%
v(t_0)=\left( 
\begin{array}{c}
0_{\mathbb{R}^{2}} \\ 
\varphi 
\end{array}%
\right), \label{6.2}
\end{equation}
where
\begin{equation*}
\mathcal{B}(t)\left( 
\begin{array}{c}
0_{\mathbb{R}} \\ 
\varphi 
\end{array}%
\right) :=\left( 
\begin{array}{c}
\int_{I}\beta _{0}(t,x)\varphi (x)dx \\ 
\int_{I}\beta _{1}(t,x)\varphi (x)dx \\ 
0_{L^{p}}%
\end{array}%
\right) 
\text{ and }
f(t):=\left( 
\begin{array}{c}
h_{0}(t) \\ 
h_{1}(t) \\ 
g(t,.)%
\end{array}%
\right) .
\end{equation*}
By using Lemma \ref{LE6.1} we know that $(\mathcal{A}+\alpha I)_0$ the part of $(\mathcal{A}+\alpha I)$ is the infinitesimal generator of $\left\lbrace T_{(\mathcal{A}+\alpha I)_0}(t) \right\rbrace_{t \geq 0}$ an analytic semigroup of bounded linear on $X_0$. By using Lemma \ref{LE6.4} we deduce that $(\mathcal{A}+\alpha I)$ generates an integrated semigroup $\left\lbrace S_{(\mathcal{A}+\alpha I)}(t) \right\rbrace_{t \geq 0}$. 
Consider for each initial time $t_0 \in \mathbb{R}$  the parabolic equation 
\begin{equation}
\left\{ 
\begin{array}{l}
\dfrac{\partial u(t,x)}{\partial t}=\dfrac{\partial
^{2}u(t,x)}{\partial x^{2}}+\alpha u(t,x),\text{ for }t\geq t_0
\text{ and }x\in (0,1), \\ 
-\dfrac{\partial u(t,0)}{\partial x}=\int_{I}\beta
_{0}(t,x)\varphi (x)dx \\ 
+\dfrac{\partial u(t,1)}{\partial x}=\int_{I}\beta
_{1}(t,x)\varphi (x)dx \\ 
u(t_0,.)=\varphi \in L^{p}(I,\mathbb{R}),%
\end{array}%
\right.   \label{6.3}
\end{equation}
this equation corresponds to the abstract Cauchy problem for each initial time $t_0 \in \mathbb{R}$ 
\begin{equation}
\dfrac{dv(t)}{dt}=(\mathcal{A}+\alpha I) v(t) +\mathcal{B}(t)v(t),\text{ for }t\geq t_0\text{ and }%
v(t_0)=\left( 
\begin{array}{c}
0_{\mathbb{R}^{2}} \\ 
\varphi 
\end{array}%
\right) .  \label{6.4}
\end{equation}

\noindent \textbf{Variation of constants formula :} By using Proposition \ref{PROP1.5} we obtain the following result. 
\begin{proposition} \label{PROP6.6} The Cauchy problem (\ref{6.4}) generates a
unique evolution family \\
$\{U_\mathcal{B}(t,s) \}_{(t,s)\in \Delta}\subset \mathcal{L}(X_0)$. Moreover $U_\mathcal{B}(\cdot,t_0)x_0\in C([t_0,+\infty),X_0)$ is the unique solution of the fixed point problem 
\begin{equation}  \label{6.5}
U_\mathcal{B}(t,t_0)x_0=T_{(\mathcal{A}+\alpha I)_0}(t-t_0)x_0+\dfrac{d}{dt}\int_{t_0}^t
S_{(\mathcal{A}+\alpha I)}(t-s)\mathcal{B}(s)U_\mathcal{B}(s,t_0)x_0ds, \ t\geq t_0.
\end{equation}
If we assume in addition that 
\begin{equation*}
\sup_{t\in \mathbb{R}} \Vert \beta _{0}(t,.) \Vert_{L^{q}}+\Vert \beta _{1}(t,.) \Vert_{L^{q}}<+\infty
\end{equation*}
then the evolution family $\{U_\mathcal{B}(t,s) \}_{(t,s)\in \Delta}$  is exponentially
bounded.
\end{proposition}
By using Theorem \ref{TH1.6} we obtain. 
\begin{theorem}
\label{TH6.7} For each $t_0\in \mathbb{R}$, each $x_0\in X_0$ and each $%
f\in C([t_0,+\infty],X)$ the unique integrated solution $v_f\in
C([t_0,+\infty],X_0)$ of (\ref{6.2}) is given for each $t\geq t_0$ by 
\begin{equation} \label{6.6}
v_f(t)=U_{\mathcal{B}}(t,t_0)x_0+\underset{\lambda \rightarrow +\infty}{\lim}
\int_{t_0}^t U_\mathcal{B}(t,s)\lambda R_\lambda(\mathcal{A}+\alpha I)f(s)ds
\end{equation}
where the limit exists in $X_0$. Moreover the convergence in (\ref{6.6}) is
uniform with respect to $t,t_0 \in I$ for each compact interval $I \subset 
\mathbb{R}$.
\end{theorem}

\noindent \textbf{Exponential dichotomy result :} By using Theorem \ref{TH5.13} we obtain de following result
\begin{theorem} \label{TH6.8}
Assume  that 
\begin{equation*}
\sup_{t\in \mathbb{R}} \Vert \beta _{0}(t,.) \Vert_{L^{q}}+\Vert \beta _{1}(t,.) \Vert_{L^{q}}<+\infty
\end{equation*}
Then the following assertions are equivalent

\begin{itemize}
\item[i)] The evolution family $\{ U_\mathcal{B}(t,s)\}_{(t,s)\in \Delta}$ has an
exponential dichotomy.

\item[ii)] For each $f\in BC(\mathbb{R},X)$, there exists a unique
integrated solution $u\in BC(\mathbb{R},X_0)$ of (\ref{6.2}).
\end{itemize}
\end{theorem}

\begin{assumption} \label{ASS6.9}
Assume that $\alpha>0$ and $\alpha \neq -  (\pi k)^2, \forall k \in \mathbb{N}$. 
\end{assumption}
Then the spectrum of $\mathcal{A}+\alpha I$ do not contain any purely imaginary eigenvalue, and by using Lemma  \ref{LE6.1} and Remark \ref{REM6.5} we deduce that 

$$
\sigma(\mathcal{A}+\alpha I)=\sigma(\mathcal{A}_0+\alpha I)=\left\lbrace -  (\pi k)^2+\alpha:k \in \mathbb{N} \right\rbrace 
$$
therefore 
$$
0 \notin \sigma(\mathcal{A}_0+\alpha I).
$$
Then $U(t,s):=T_{\mathcal{A}+\alpha I}(t-s)$ has an exponential dichotomy and we can apply Theorem \ref{TH5.14} with $A+B(t):=\mathcal{A}+\alpha I$ and $C(t):=\mathcal{\mathcal{B}}(t)$. 
\begin{theorem}  \label{TH6.10}
Let Assumption \ref{ASS6.9} be satisfied. There exists $\varepsilon>0$ such that
\begin{equation*}
\sup_{t\in \mathbb{R}} \Vert \beta _{0}(t,.) \Vert_{L^{q}}+\Vert \beta _{1}(t,.) \Vert_{L^{q}}<\varepsilon
\end{equation*}
implies that the evolution family $\{U_\mathcal{B}(t,s) \}_{(t,s)\in \Delta}\subset \mathcal{L}(X_0)$ has an exponential dichotomy.
\end{theorem}

\end{document}